\documentclass[a4paper, 11pt]{amsart}
\usepackage{amsmath, amssymb, amsthm, amsfonts,enumitem,tikz,relsize}
\usepackage[normalem]{ulem}
\usepackage[colorlinks = true, citecolor = blue, hypertexnames=false]{hyperref}
\usepackage{xcolor}
\hypersetup{
    colorlinks,
    linkcolor={red!70!black},
    citecolor={blue!70!black},
    urlcolor={blue!80!black}
}
\usepackage{framed,tabu,float,caption,scalerel}
\usepackage[OT1]{fontenc}
\usepackage{comment}
\usepackage[prependcaption,textsize=tiny, backgroundcolor=blue!20!white,linecolor=blue!80]{todonotes}

\theoremstyle{plain}

\newtheorem*{maintheorem}{\hypertarget{h:theorem}{Theorem}}
\newtheorem*{corollary*}{Corollary}

\newtheorem*{uproperties}{\hypertarget{h:uproperties}{Universe Properties}}
\newtheorem*{dproperties}{Dimension Properties}
\newtheorem*{divproperties}{\hypertarget{h:divproperties}{Divisibility Properties}}
\newtheorem*{cproperties}{Connectedness Properties}
\newtheorem*{chlemma}{\hypertarget{h:chlemma}{Characteristic Lemma}}
\newtheorem*{clemma}{\hypertarget{h:clemma}{Coprimality Lemma}}
\newtheorem*{wlemma}{\hypertarget{h:wlemma}{Weight Lemma}}
\newtheorem*{rlemma}{\hypertarget{h:rlemma}{Recognition Lemma}}
\newtheorem*{elemma}{\hypertarget{h:elemma}{Extension Lemma}}
\newtheorem*{glemma}{\hypertarget{h:glemma}{Geometrisation Lemma}}
\newtheorem*{fglemma}{\hypertarget{h:fglemma}{First Geometrisation Lemma}}

\theoremstyle{definition}
\newtheorem*{definition*}{Definition}
\newtheorem*{fact*}{Fact}
\newtheorem*{example*}{Example}
\newtheorem*{counterexample*}{Counterexample}
\newtheorem*{remark}{Remark}
\newtheorem*{remarks}{Remarks}
\newtheorem*{examples}{Examples}

\newtheorem*{conjecture*}{Conjecture}
\newtheorem*{question*}{Question}
\newtheorem*{problem*}{Problem}
\newtheorem*{notation}{Notation}
\theoremstyle{remark}
\newtheorem{claim}{Claim}

\newenvironment{proofclaim}[1][Proof of Claim]{\begin{proof}[#1]}{\end{proof}}

\newcommand{\qoplus}{\mathbin{\scalerel*{\left(\mkern-2mu+\mkern-2mu\right)}{\bigcirc}}}
\DeclareMathOperator*{\bigqoplus}{\left(\mkern-2mu\scalerel*{+}{\bigoplus}\mkern-2mu\right)}

\DeclareMathOperator{\ad}{ad}
\DeclareMathOperator{\tr}{tr}
\DeclareMathOperator{\im}{im}

\DeclareMathOperator{\charac}{char}
 
\DeclareMathOperator{\Aut}{Aut}

\DeclareMathOperator{\End}{End} 
\DeclareMathOperator{\Sym}{Sym} 
\DeclareMathOperator{\Alt}{Alt}
\DeclareMathOperator{\ustd}{std}
\DeclareMathOperator{\rstd}{\overline{std}}
\DeclareMathOperator{\sgn}{sgn}
\DeclareMathOperator{\perm}{perm}

\DeclareMathOperator{\Out}{Out}
\DeclareMathOperator{\GL}{GL}
\DeclareMathOperator{\PSL}{PSL} 
\DeclareMathOperator{\PGL}{PGL} 
\DeclareMathOperator{\SL}{SL}

\DeclareMathOperator{\HSP}{HSP}

\newcommand{\cU}{\mathcal{U}}
\newcommand{\Ab}{\mathbf{Ab}}
\DeclareMathOperator{\Ar}{Ar}
\DeclareMathOperator{\Ob}{Ob}

\DeclareMathOperator{\Id}{Id}

\newcommand{\Mod}{\mathbf{Mod}}
\newcommand{\generated}[1]{\left\langle#1\right\rangle}

\newcommand{\bF}{\mathbb{F}}

\newcommand{\bK}{\mathbb{K}}

\newcommand{\bN}{\mathbb{N}}
\newcommand{\bQ}{\mathbb{Q}}

\newcommand{\bZ}{\mathbb{Z}}
\newcommand{\bC}{\mathbb{C}}

\begin{document}
\title[$\Sym(n)$- and $\Alt(n)$-modules]{$\Sym(n)$- and $\Alt(n)$-modules with an additive dimension\\ (Paris Album No.~2)
}
\author{Luis Jaime Corredor}
\address{Departamento de Matem\'aticas, Universidad de los Andes, Bogot\'a, Colombia}
\email{lcorredo@uniandes.edu.co}
\author{Adrien Deloro}
\address{Sorbonne Universit\'e, Institut de Math\'ematiques de Jussieu-Paris Rive Gauche, \textsc{cnrs}, Universit\'e Paris Diderot
and \'Ecole Normale Sup\'erieure-\textsc{psl}, D\'epartement de Math\'ematiques et Applications}
\email{adrien.deloro@imj-prg.fr}
\author{Joshua Wiscons}
\address{Department of Mathematics and Statistics\\
California State University, Sacramento\\
Sacramento, CA 95819, USA}
\email{joshua.wiscons@csus.edu}
\date{\today}
\keywords{symmetric groups, finite-dimensional algebra, first-order representation theory}
\subjclass[2020]{Primary 20C30; Secondary 03C60, 20F11}




\newpage

\begin{abstract}
We revisit, clarify, and generalise classical results of Dickson and (much later) Wagner on minimal $\Sym(n)$- and $\Alt(n)$-modules.
We present a new, natural notion of `modules with an additive dimension' covering at once the classical, finitary case as well as modules definable in an $o$-minimal or finite Morley rank setting; 
in this  context, we fully identify the faithful $\Sym(n)$- and $\Alt(n)$-modules of least dimension.
%
\end{abstract}

\maketitle



\vspace{-10pt}
\begin{center}
\S~\ref{S:introduction}. Introduction \quad --- \quad \S~\ref{S:context}. The context \quad --- \quad \S~\ref{S:proof}. The proof 
\end{center}
\vspace{5pt}


\section{Introduction}\label{S:introduction}

This work belongs to the topic of \emph{first-order representation theory}, i.e.~representation theory viewed through an elementary lens.
Here the focus in on the category of explicitly constructible objects, or what mathematical logic calls the \emph{definable} category; a consequence is that we avoid any reference to characters.
This is not motivated by the mere `purity of methods' but by questions in model theory.
The topic is naturally emerging out of several recent works, 
rooted in the 2008 article of Alexandre Borovik and Gregory Cherlin \cite{BoCh08} and pushed further by papers such as \cite{DeA09,BoDe15,BoBe18, BoA20, BeBo21}.
(Another road to the effective understanding of geometric algebra is \emph{black box algebra} as in \cite{BY18} and ongoing work.)
We stress that, though inspired by model theory, the present article is likely to be of broad interest; no exposure to model theory is required to understand our work.


We give a significant expansion and clarification of a classical result by Leonard Dickson from 1908 on the minimal linear representations of symmetric groups. 
Our \hyperlink{h:theorem}{Theorem} identifies the minimal faithful representations of $\Sym(n)$ and $\Alt(n)$ on finite or infinite abelian groups in the presence of a rudimentary notion of dimension but no a priori---and often no a posteriori---vector space structure. This may be viewed as a natural evolution of linear representation theory, one which focuses more on `elementary' properties (e.g.~generators and relations) and less on higher structure. In the case of algebraic groups, which we keep in mind for the future, our point of view would be quite in the spirit of the Chevalley-Steinberg approach.
We stress that our context does \emph{not} allow for character theory nor even Maschke's Theorem, making matters non-trivial though basic.

All that remains from the usual linear theory is a loose form of dimensionality.
The study 
of structures whose definable sets are equipped with one of various 
notions of dimension is a central theme in model theory, and our work here treats numerous classes at once---including groups of finite Morley rank and $o$-minimal groups, but also finite groups---in a common, natural, and new setting.

Our original motivation was a set of concrete model-theoretic problems.
One such is an application to permutation groups possessing a high degree of \emph{generic} transitivity. 
A study of these was initiated in the setting of groups of finite Morley rank in work by Borovik and Cherlin where they posed the 
problem of showing that generic $(n+2)$-transitivity on a set of Morley rank $n$ implies that the group is $\PGL_{n+1}(\bK)$ in its natural action on $\mathbb{P}(\bK)$. More information and explicit connections to the present work can be found in \cite{BoCh08} as well as in the 2018 paper of Tuna Alt\i{}nel and the third author \cite{AlWi18}, which solves the $n=2$ case. The present work  grew out of the third author's desire to generalise the $n=2$ approach to $n\ge 3$, but the topic turned out extremely interesting in its own right.


\subsection{The result}\label{s:results}
Our main result generalizes a century-old theorem by Dickson \cite{DiL08} and its more recent companions \cite{WaA76,WaA77}; it also corrects and expands on \cite[Lemma~4.6]{BoCh08}. In doing so, our treatment handles simultaneously the finite and the `tame infinite', in a sense that model theory seeks to carve out.

We study $\Sym(n)$ and $\Alt(n)$-modules $V$ that carry a basic notion of dimension on certain groups (and quotients) definable---in the logical sense---from $V$ and the acting group. To the logician we must stress that our dimension need not apply to all definable \emph{sets}; we only require it to be defined on the intersection of the `definable universe' in the sense of logic and the `variety generated by $V$' in the sense of universal algebra.
This intuition is  axiomatised in Section~\ref{S:context} via the definitions of a \emph{modular universe} and an \emph{additive dimension}; the relevant notions of connectedness (\emph{dim-connectedness}) and irreducibility (\emph{dc-irreducibility}) are also key. 

As one expects, the `characteristic' of a module is an important parameter: $V$ is said to have \emph{prime characteristic $p$} if it has exponent $p$ and \emph{characteristic $0$} if it is divisible. This  definition allows for modules without a well-defined characteristic such as $\bZ/12\bZ$; it also allows for torsion modules of characteristic $0$ such as the Pr\"ufer quasi-cyclic groups $C_{p^\infty}$.

In the classical setting, the minimal faithful representations for $\Sym(n)$ and $\Alt(n)$ are canonical and easy to construct, assuming $n$ is large enough.
Among other places, they appear in \cite{DiL08}, but we briefly describe them here. This also gives us the opportunity to introduce  notation.

\begin{notation}[standard module]
Let $S:=\Sym(n)$.
\begin{enumerate}
\item
Let $\perm(n,\bZ) = \bZ e_1\oplus \cdots \oplus \bZ e_n$ be the $\bZ[S]$-module with $S$ permuting the $e_i$ naturally. 
There are two obvious submodules:
\begin{itemize}
\item 
$\ustd(n, \bZ) := [S,\perm(n, \bZ)] = \{\sum_i c_i e_i : c_i\in \bZ \text{ and } \sum c_i = 0\}$; 
\item
$Z(\perm(n, \bZ)) := C_{\perm(n,\bZ)}(S) = \{\sum_i ce_i : \text{$c\in \bZ$}\}$, 
\end{itemize}
using usual notation for commutators and centralisers.
Over $\bZ$ these are disjoint but not so in general.
\item 
For any abelian group $L$ (considered as a trivial $S$-module), define:
\begin{itemize}
\item 
$\perm(n,L) := \perm(n,\bZ) \otimes_\bZ L$;
\item
$\ustd(n,L) := [S,\perm(n, L)] = \ustd(n,\bZ) \otimes_\bZ L$; 
\item
$Z(\ustd(n, L)) := C_{\ustd(n,L)}(S)$.
\end{itemize}
We arrive at the canonical subquotient:
\begin{itemize}
\item 
$\rstd(n, L) = \ustd(n, L)/Z(\ustd(n, L))$,
\end{itemize}
which  we refer to as the (reduced) \emph{standard module} over $L$.
\item When $L=C_k$ is cyclic of order $k$, we simply write $\perm(n, k)$, $\ustd(n, k)$,  and $\rstd(n, k)$.
\end{enumerate}
\end{notation}

\begin{remarks}\leavevmode
\begin{itemize}
\item Notice how $Z(\ustd(n,L)) = \{\sum e_i\otimes a : a\in \Omega_n(L)\}$, so $\rstd(n, L)$ differs from $\ustd(n,L)$ only when $\Omega_n(L)\neq 0$. (Here $\Omega_n(L)$  denotes the subgroup of elements of order dividing $n$.)
\item The module $\perm(n,L)$ may be realised as $L^n$ under $\sigma(a_1,\ldots,a_n) = (a_{\sigma^{-1}(1)},\ldots,a_{\sigma^{-1}(n)})$; it is easily constructed from $L$ together with the action of each element of $S$. 
Model theorists will recognize this as an \emph{interpretable object}, which we simply call \emph{definable}.
\end{itemize}
\end{remarks}

The classical setting is $\ustd(n,p)$ for $p$ a prime. Here, $\ustd(n,p)$ is irreducible of dimension $n-1$ whenever $p\nmid n$  
(with the same true for $\ustd(n, \bQ)$). However, when $p\mid n$ and $n \geq 5$, the module $\rstd(n, p)$ is faithful and irreducible of dimension $n-2$, a point which fails of $\rstd(4, 2)$.

Less classical is the following example regarding actions on tori.

\begin{example*}
Notice how a maximal torus of $\GL_{n}(\bC)$ is a $\Sym(n)$-module via the action of the Weyl group. 
As $\Sym(n)$-modules, one finds that
\begin{itemize}
\item $\perm(n,\bC^\ast)$ is isomorphic to a maximal torus of $\GL_{n}(\bC)$; 
\item $\ustd(n,\bC^\ast)$ is isomorphic to a maximal torus of $\SL_{n}(\bC)$;  
\item $\rstd(n,\bC^\ast)$ is isomorphic to a maximal torus of $\PSL_{n}(\bC)$. 
\end{itemize}
The various finite subgroups of $n^\text{th}$-roots of unity  yield finite submodules $C \le Z(\ustd(n, \bC^\ast))$, each naturally isomorphic to some $Z \le Z(\SL_{n}(\bC^\ast))$. In each case, $\ustd(n, \bC^\ast)/C$ is isomorphic to the maximal torus of $\SL_{n}(\bC)/Z$. The modules $\ustd(n, \bC^\ast)/C$ all satisfy the relevant notion of irreducibility here (dc-irreducibility) and must be accounted for in our classification. 
\end{example*}

We now state our main result. In what follows, $\Mod(G, d, q)$ denotes the class of all $G$-modules (\S~\ref{s:universe}) that carry an additive dimension (\S~\ref{s:dimension}) and are dim-connected (\S~\ref{s:dc}) of dimension $d$ and characteristic $q$ (\S~\ref{s:characteristic}). Also, $V\in \Mod(G, d, q)$ is \emph{dc-irreducible} (\S~\ref{s:dcirreducible}) if $V$ contains no non-trivial, proper, dim-connected $G$-submodule. Details are in \S~\ref{S:context};  the proof is in \S~\ref{S:proof}.

\begin{maintheorem}\label{t:main}
Let $G = \Alt(n)$ or $\Sym(n)$. Suppose $V \in \Mod(G, d, q)$ is faithful and dc-irreducible with $d < n$.
Assume $n\ge 7$; if $G = \Alt(n)$ and $q=2$, further assume that $n \geq 10$. 

Then there is a dim-connected submodule $L \le V$ such that the structure of $V$ falls into one of the following cases:
\begin{center}
\tabulinesep = 1.2mm
\begin{tabu} {X[0.8,c,m]|[1pt]X[0.3,c,m]|[1pt]X[2,c,m]}
$q$  & $d$   & Structure of $V$ \\ \tabucline[2pt]{-}
$q > 0$ and $q\mid n$ & $n-2$ & isomorphic to $\rstd(n, L)$ or $\sgn \otimes\,   \rstd(n, L)$ \\ \tabucline[1pt]{-}
$q > 0$ and $q\nmid n$ & $n-1$ & isomorphic to $\rstd(n, L)$ or $\sgn \otimes  \rstd(n, L)$ \\ \tabucline[1pt]{-}
$q = 0$ & $n-1$ & covered by $\ustd(n, L)$ or $\sgn \otimes \ustd(n, L)$ 
\end{tabu}
\end{center}
Moreover, when $q = 0$, the kernel of the covering map is 
$\langle{\sum_{i = 1}^{n-1} (e_i-e_n)}\rangle\otimes K$, in usual notation, for some $K \leq \Omega_n(L)$.
\end{maintheorem}

\begin{remarks}\hfill
\begin{itemize}
\item Note that $q =\charac L$, and if $q>0$, $\rstd(n, L)$ 
is completely reducible as  $\bigoplus_X \rstd(n, q)$ for $X$ an $\bF_q$-basis of $L$.  The situation when $q=0$ is complicated by tori like those given in the  example above.
\item The restrictions on $n$ are optimal. For example, in characteristic $3$, one has $\Alt(6)\simeq \PSL_2(\bF_9)$ with the  adjoint representation in dimension $3$. In characteristic $2$,  $\Alt(9)$ has three faithful representation over $\mathbb{F}_2$ of (least) degree $8$~\cite{atlas3}. 
The two exceptional representations were missed in \cite[(4.5)~Lemma]{WaA76};
this was corrected by G.D.~James in  \cite[Theorem~6]{JaG83}.

In our more general setting, one can still establish the lower bound of $8$ on the dimension of a faithful $\Alt(9)$-module in characteristic $2$, but we do not achieve (nor even try for) identification. (See the remark following the proof of the \hyperlink{h:glemma}{Geometrisation Lemma}.)
\end{itemize}
\end{remarks}

\subsection{Lingering questions}\label{s:questions}

The \hyperlink{h:theorem}{Theorem} places natural restrictions on $n$, but in fact,
 the minimal dimension of a faithful $\Sym(n)$-module is indeed as expected \emph{for all $n$} as a consequence of our \hyperlink{h:fglemma}{First Geometrisation Lemma}, where we also identify those of dimension $(n-2)$. However, identification of the dc-irreducible modules in $\Mod(\Sym(5), 4)$ and $\Mod(\Sym(6), 4)$ remains open.
(Do note that $\Mod(\Sym(5), 4, 2)$ contains irreducible modules coming from the so-called Specht module for the partition $(3,2)$. See for example \cite[5.2~Example]{JaG78}.) 

Identification of the minimal faithful $\Alt(n)$-modules for small $n$ is also open. 
Reconstructing the adjoint action of $\Alt(6) \simeq \PSL_2(\bF_9)$ is a problem of particular interest. One also has $\Alt(8)\simeq \SL_4(\bF_2)$ with the natural action as well as the particularly exceptional $\Alt(5)$: it appears as $\SL_2(\bF_4)$ in characteristic $2$, as $\PSL_2(5)$ in characteristic $5$, and as the symmetries of a regular icosahedron in all other characteristics (over a field where $5$ is a square). The following table summarizes the conjectural lower bounds.

\begin{table}[H]
{\small
\tabulinesep = 1.2mm
\begin{tabu} {X[1.1,$c,m]|[2pt]X[$c,m]|[1pt]X[$c,m]|[1pt]X[$c,m]|[1pt]X[$c,m]|[1pt]X[$c,m]}
 p      & 2   & 3   & 5   & 7   & >7\text{ or } 0 \\ \tabucline[2pt]{-}
\Alt(5) & 2   & 3   & 3   & 3   & 3   \\ \tabucline[1pt]{-}
\Alt(6) & 4   & 3   & 5   & 5   & 5   \\ \tabucline[1pt]{-}
\Alt(7) & 4   & 6   & 6   & 5   & 6   \\ \tabucline[1pt]{-}
\Alt(8) & 4   & 7   & 7   & 7   & 7
\end{tabu}
}
\caption*{Conjectural minimal dimension for faithful $\Alt(n)$-modules in characteristic $p$ with small $n$.}
\label{tab.SmallDimension}
\end{table}

\begin{problem*}
Identify the minimal faithful $\Sym(n)$- $\Alt(n)$-modules for all  $n$.
\end{problem*}

Of course, one could target  higher-dimensional  dc-irreducible $\Sym(n)$- $\Alt(n)$-modules, but this appears to be out of reach at present. However, simply identifying a reasonable lower bound for the dimension of the `second smallest' dc-irreducible module would be welcome. Following the classical case, we expect something along the lines of $n(n-5)/2$ (see~\cite{JaG83}).

\begin{problem*}
Let $G = \Alt(n)$ or $\Sym(n)$ with $n$ sufficiently large. Prove that if $V \in \Mod(G, d, q)$ is faithful and dc-irreducible with $d < n(n-5)/2$, then up to tensoring with the signature, $V$ is standard (i.e.~the structure of $V$ is as in the  \hyperlink{h:theorem}{Theorem}) with $(n-2)\dim L \le d\le  (n-1)\dim L$ and intermediate values are possible only when $q=0$.
\end{problem*}

As should be clear, we are operating under the conjectural principle that although our context is quite general, the minimal objects still fall into the familiar linear-algebraic setting, a principle well-aligned with the recent work of Borovik \cite{BoA20}. We quite believe in this and would thus love to see a counter-example to shatter our illusion.

The problem of determining the minimal dimension of a group carrying a faithful action of $\Sym(n)$ or $\Alt(n)$ seems both interesting and relevant in the nonabelian case as well. With additional definability/compatibility hypotheses, the soluble case can easily be controlled. For nonsoluble groups, we propose the following crude bound, which is likely far from optimal.

\begin{problem*}
Let $H$ be a nonsoluble group on which $\Alt(n)$ acts faithfully and definably by automorphisms. Suppose there is a nonabelian notion of dimension, say Morley rank, making $H$ dim-connected. Show $\dim H \ge n$ for sufficiently large $n$.
\end{problem*}

Notice that low values of $n$ will complicate the picture even for $\Sym(n)$: for example, $\Sym(5)\simeq \PGL_2(5)$, which can be construed as $3$-dimensional.

For the present article we however stick to the abelian case. The proof of the \hyperlink{h:theorem}{Theorem} will be in \S~\ref{S:proof}; we first turn to the general setting.
\section{The context}\label{S:context}
We now give the  setting for our study of $\Sym(n)$- and $\Alt(n)$-modules. In addition to defining modules equipped with an additive dimension, we also present notions of connectedness, irreducibility, and the characteristic. In short, the goal of this section is to fully explain the phrase `let $V \in \Mod(G, d, q)$ be dc-irreducible'. 

The landscape will likely be both familiar and surprising to the reader versed in model theory. We seek a notion of dimension that encompasses simultaneously the linear dimension over $\bF_p$ for finite representations as well as model-theoretic dimensions (e.g.~Morley rank) for  infinite representations. The context we present is extremely natural, yet looks new to us. (Unfortunately, somewhat conflicting terminology with Frank Wagner's recent `dimensional groups' was unavoidable~\cite{WaF20}.)

We first define modules (\S~\ref{s:universe}) with an additive dimension (\S~\ref{s:dimension}), and the notion of dim-connectedness (\S~\ref{s:dc}). We discuss the characteristic of a module (\S~\ref{s:characteristic}) and then introduce classes $\Mod(G, d, q)$ as well as the relevant notion of irreducibility (\S~\ref{s:dcirreducible}). The overview concludes with our key tool: an expected \hyperlink{h:clemma}{Coprimality Lemma} (\S~\ref{s:coprimality}).

\subsection{Modular universes and modules}\label{s:universe}

A balance is difficult to strike between categorical generality and model-theoretic care for elementary constructions.
We opted for a categorical vision but avoided any specialised language. We believe the categorist will instantly grasp the context, and the logician will readily check that it generalises definable universes. Do note that we use $\ker f$ and $\im f$ to refer to the kernel and image of $f$ in the \textit{algebraic} sense.

\begin{definition*}
A \emph{modular universe} is a subcategory $\cU$ of the category $\Ab$ of abelian groups satisfying the following closure properties.
\begin{itemize}
\item{}[\textsc{inverses}]
If $f \in \Ar(\cU)$ is an isomorphism, then $f^{-1} \in \Ar(\cU)$.
\item{}[\textsc{products}] If $V_1, V_2 \in \Ob(\cU)$ and $f_1, f_2 \in \Ar(\cU)$, then 
$V_1 \times V_2 \in \Ob(\cU)$, and 
	$\Ar(\cU)$ contains $f_1 \times f_2$,  the projections $\pi_i : V_1\times V_2 \rightarrow V_i$, and the diagonal embeddings $\Delta_k:V_1\rightarrow V_1^k$.
\item{}[\textsc{sections}] 
If $W \leq V$ are in $\Ob(\cU)$, then $V/W\in \Ob(\cU)$ and $\Ar(\cU)$ contains the inclusion $\iota:W\rightarrow V$ and quotient $p:V\rightarrow V/W$ maps.
\item{}[\textsc{kernels/images}] 
If $f:V_1 \rightarrow V_2$ is in $\Ar(\cU)$, then $\ker f, \im f \in \Ob(\cU)$, and for all $W_1,W_2 \in \Ob(\cU)$,
	\begin{itemize} 
	\item
	if $W_1 \leq \ker f$, the induced map $\overline{f}\colon V_1/W_1 \to V_2$ is in $\Ar(\cU)$;
	\item 
	if $\im f \le W_2 \le V_2$,  the induced map $\check{f}\colon V_1 \to W_2$ is in $\Ar(\cU)$.
	\end{itemize} 
\item{}[\textsc{module structure}]
If $V \in \Ob(\cU)$, then $\Ar(\cU)$ contains the sum  map $\sigma \colon V\times V \to V$ and the multiplication-by-$n$  maps $\mu_n\colon V \to V$.
\end{itemize}
The objects of a modular universe $\cU$ are called its \emph{modules} and the arrows its \emph{compatible morphisms}.
\end{definition*}

%

\begin{remarks}\label{r:catenarity}\leavevmode
\begin{itemize}
\item
We could not find an official categorical name for our setting. `Topologising, abelian subcategory' does not suffice; the axiom of inverses is of importance to us, but we won't go as far as assuming that the subcategory is replete.
\item
The axioms immediately allow for restrictions of compatible maps and the computing of inverse images. Indeed, if $f:V_1\rightarrow V_2$ is in $ \Ar(\cU)$ with $W_1 \le V_1$ in $\Ob(\cU)$, then the restriction of $f$ to $W_1$ is $f\circ \iota$ for $\iota :W_1 \rightarrow V_1$ the inclusion. And for $W_2 \le V_2$  in $\Ob(\cU)$, $f^{-1}(W_2)$ is the kernel of $p\circ f$ for $p:V_2 \rightarrow V_2/W_2$ the quotient map.
\item
Model theorists might expect a `characterisation of arrows' which we do \emph{not} require: $f \in \Ar(\cU)$ if and only if its domain, image and graph are in $\Ob(\cU)$.
\end{itemize}
\end{remarks}

\begin{examples}\leavevmode
\begin{itemize}
\item
From algebra: the category of all abelian groups or of all abelian $p$-groups (equipped with all group morphisms) forms a modular universe; both are full subcategories of $\Ab$.
An important variation is the category of all abelian $p$-groups \emph{of finite Pr\"ufer $p$-rank}. (Such groups can contain only \emph{finite} powers of the quasi-cyclic group $C_{p^\infty}$; the maximal such power is called the $p$-rank.)

For a given ring $R$, the category of $R$-modules (equipped with all $R$-morphisms) is a modular universe.
\item
From Lie theory: the collection of all abelian Lie groups (with Lie morphisms) forms a modular universe after adjusting the notion of $\leq$ (which has to be closed).
\item
From universal algebra: the variety generated by a given abelian group $V$ (with group morphisms) forms a modular universe; this can be computed as $\HSP(V)$, which is the collection of all Homomorphic images of Subgroups of Products of $V$.
\item
From model theory: the abelian part of the `interpretable universe' (which we call definable) forms a modular universe; specifically, this is looking at the category of all abelian groups definable in some first-order theory equipped with the definable group morphisms.
\item
If $V$ is a module of a universe $\cU$, then there is a smallest subuniverse containing $V$. This is contained in both $\HSP(V)$ and the abelian-definable universe $\cU_{\rm{def}}$ from model theory. Notice that the intersection $\HSP_{\rm{def}}(V) = \HSP(V) \cap \cU_{\rm{def}}(V)$ might be substantially smaller than $\cU_{\rm{def}}$ if we add a highly extrinsic, secondary abelian structure on some non-definable subsets of $V$. (Typically, if $\bK$ is one of the pathological fields model theorists adore, then $\HSP_{\rm{def}}(\bK_+)$ will miss all the exotic structure.)
\end{itemize}
\end{examples}

Modular universes allow for a variety of additional constructions. We highlight several important ones after giving the relevant notion of a $G$-module for our setting. 

\begin{definition*}
Suppose $V$ is a module in the universe $\cU$. If a group $G$ acts on $V$ by compatible morphisms of $\cU$, we say that $V$ is a \emph{$G$-module} in $\cU$. 
\end{definition*}

In model-theoretic terms, we are assuming that $G$ acts by definable automorphisms, but we are \emph{not} assuming definability of the action, viz.~definability of the triple $(G, V, \cdot)$. For instance, $G$ itself need not be definable.


\begin{uproperties}
Let $\cU$ be a modular universe.
\begin{enumerate}[label=(\roman*)]
\item{}
\textsc{[meet/join]} If $V_1, V_2 \leq V$ are modules in $\cU$, then so are $V_1 \cap V_2$ and $V_1 + V_2$.
\item{}
\textsc{[extensions]} Suppose $V_1, V_2 \leq V$ are modules in $\cU$ with $V_1 + V_2 = V$. If $f_i \colon V_i \to W$ are compatible morphisms that agree on $V_1 \cap V_2$, then there is a compatible $f\colon V \to W$ extending both.
\item{}
\textsc{[permutation modules]} If $V$ is a module in $\cU$, then $\perm(n, V)$ is a $\Sym(n)$-module in $\cU$.
\item{}
\textsc{[enveloping algebras]} If $V$ is a $G$-module in $\cU$ for some group $G$, then each map in the subring of $\End(V)$ generated by $G$ is compatible. 
\end{enumerate}
\end{uproperties}
\begin{proof}
Assume $V_1, V_2 \leq V$ are modules in $\cU$. Let $\sigma: V\times V\rightarrow V$ be the addition map and $\Delta_k: V \rightarrow V^k$ the diagonal embedding.
\begin{enumerate}[label=(\roman*)]
\item 
Note that $V_1 \cap V_2 = \Delta_2^{-1}(V_1\times V_2)$ and $V_1 + V_2 = \sigma(V_1\times V_2)$. 
%
\item 
Let $g$ be the restriction of $\sigma$ to $V_1\times V_2$, and set $I := \{(a, -a): a \in V_1 \cap V_2\} = \ker g$.  Then the induced isomorphism $\overline{g}\colon V_1 \times V_2/I \rightarrow V$ is in $\Ar(\cU)$, so $\overline{g}^{-1}$ is as well. Define $h: V_1\times V_2 \rightarrow V$ via restriction of $\sigma \circ (f_1\times (f_2\circ\mu_{-1}))$ to $V_1\times V_2$, where $\mu_{-1}:V\rightarrow V$ is inversion; then $h$  computes $f_1 - f_2$. Since $I \le \ker h$, the induced map $\overline{h}\colon V_1 \times V_2/I \rightarrow V$ is in $\Ar(\cU)$, and $f:=\overline{h}\circ \overline{g}^{-1}\in \Ar(\cU)$ extends both $f_1$ and $f_2$.
\item
As a set, we may identify $\perm(n, V)$ with $V^n \in \Ob(\cU)$ and the canonical summands $Ve_i$ with $\bigcap_{j\neq i} \ker \pi_k$ (for $\pi_k:V^n \rightarrow V$ the $k^\text{th}$-projection). Now, for $\alpha \in \Sym(n)$, viewed as the automorphism of $V^n$ permuting the coordinates naturally, we must show $\alpha \in \Ar(\cU)$. Let $\Delta$ be the diagonal embedding of $V^n$ into $(V^n)^n$.  Then $\alpha = (\pi_{\alpha^{-1}(1)} \times \cdots\times\pi_{\alpha^{-1}(n)})\circ\Delta.$
\item 
The (extended) sum map $\sigma_k: V^k\rightarrow V$ may be inductively defined as $\sigma\circ(\sigma_{k-1} \times \Id_V)$ so is in $\Ar(\cU)$. Further, $\Ar(\cU)$ contains the multiplication-by-$n$ maps $\mu_n \colon V \to V$ for each $n \in \bZ$. Thus, for $g_1,\ldots g_k \in G$ and $n_1,\ldots,n_k \in \bZ$, the image of $\sum n_ig_i$ in $\End(V)$ is $ \sigma_k\circ [(\mu_{n_1}\circ g_1)\times \cdots \times (\mu_{n_k}\circ g_k)] \circ \Delta_k \in \Ar(\cU)$.
\qedhere
\end{enumerate}
\end{proof}

\subsection{Additive dimensions}\label{s:dimension}

\begin{definition*}
Let $\cU$ be a modular universe.
An \emph{additive dimension} on $\cU$ is a function $\dim\colon \Ob(\cU) \to \bN$ such that for all $f\colon V \to W$ in $\Ar(\cU)$, 
\[\dim V = \dim \ker f + \dim \im f.\]
This property will be called \emph{additivity}.
\end{definition*}

\begin{examples}
Each of the following has an additive dimension:
\begin{itemize}


\item
the universe of finite-dimensional vector spaces over a fixed field, equipped with the linear dimension;
\item
the universe of abelian $p$-groups of finite Pr\"ufer $p$-rank, with dimension the $p$-rank;
\item
the universe of all abelian groups definable in a theory of finite Morley rank, with dimension the Morley rank;
\item 
the universe of all abelian groups definable in an $o$-minimal structure, equipped with $o$-minimal dimension;
\item
the universe of abelian Lie groups, equipped with Lie (manifold) dimension.
\end{itemize}
\end{examples}

\begin{remarks}
The following remarks are better understood in relation to say, behaviour of Morley rank, but can be read independently.
\begin{itemize}
\item
No assumption is made on when the dimension increases nor how; as a matter of fact, any multiple of a dimension function is again one.
\item
There need not be a descending chain condition (`\textsc{dcc}') on objects. This calls for a modified notion of connectedness in \S~\ref{s:dc}.
\item
No relationship between finiteness and $0$-dimensionality is implied. Thus, we handle in the same operational setting $\bF_p^n$ ($n$-dimensional, in finite group theory) and $\bZ$ ($0$-dimensional, in Lie theory).
\item
It is unclear whether the non-abelian case is restrictive enough for a general theory to emerge. Typically, handling commutators requires the dimension to be defined on subsets, not only subgroups. 

For example, we do not know about free groups or Tarski monsters; the question seems to be of interest but beyond our expertise.
\item
We briefly mused on the possibility of determining natural criteria for a pre-dimension---defined only on submodules of some $V$---to extend to a genuine dimension on subquotients as well, but this remains mostly unexplored. Such topics are nontrivial in model theory.
\end{itemize}
\end{remarks}

%
%

\begin{remarks}[for model theorists]
Our work stems from model theory but we wish to stress a couple of differences.
\begin{itemize}
\item
The focus is on \emph{group} subquotients instead of general definable sets. (One could define $\dim$ on cosets but we will not need that.)
\item
We work with a dimension on a  universe containing a \emph{fixed} group structure and make no demands on its behaviour in elementary extensions.
Thus dimension is not required to be a `strong' invariant (viz.~a property of the theory); as a matter of fact, an elementary extension need not bear a dimension function.

Morley rank and $o$-minimal dimension are strong invariants.
We do not have an example of (the definable universe of) an abelian group $V$ carrying an additive dimension and an elementary extension $V^\ast$ not admitting one.
%

Arguably we touch here the difference between first-order (definability in one structure) and model-theoretic (definability in family, viz.~in elementary extensions) properties; or between model-theoretic algebra and model-theory properly speaking.
\end{itemize}
\end{remarks}

\begin{dproperties}
Let $(\cU, \dim)$ be a modular universe with an additive dimension, from which we take modules.
\begin{enumerate}[label=(\roman*)]
\item
$\dim \{0_V\} = 0$.
\item
If $W\leq V$, then $\dim W \leq \dim V$.
\item
$\dim (V_1 \times V_2) = \dim V_1 + \dim V_2$.
\item
If $V_1, V_2 \leq V$, then $\dim (V_1 + V_2) = \dim V_1 + \dim V_2 - \dim (V_1\cap V_2)$.
\end{enumerate}
\end{dproperties}
\begin{proof}
For the first, apply additivity to the identity map $\iota:V\rightarrow V$.
The second follows from additivity of the quotient map $p:V\rightarrow V/W$ (and that  $\dim$ is non-negative). For the third, additivity of the projection $\pi_1 : V_1 \times V_2 \rightarrow V_1$ shows $\dim(V_1 \times V_2) = \dim V_1 + \dim(\ker \pi_1)$, so the result then follows from the fact that $\{0_{V_1}\}\times V_2$ and $V_2$ are compatibly isomorphic. For the final point, apply additivity to the restriction of $\sigma$ to $V_1\times V_2$ (using that the kernel is compatibly isomorphic to $V_1\cap V_2$).
\end{proof}

\begin{definition*}
If $(\cU, \dim)$ is a modular universe with an additive dimension and $V \in \Ob(\cU)$, we call $V$ \emph{a module with an additive dimension}, leaving $\cU$ implicit from context. 
\end{definition*}

In practice, we often have in mind $\HSP_{\rm def}(V)$ (viz.~those homomorphic images of submodules of powers of $V$ which are model-theoretically definable) when we say $V$ is a module with an additive dimension.

\subsection{dim-connectedness}\label{s:dc}
We now present a weak form of connectedness, which is \emph{not} the classical one in mathematical logic.

\begin{definition*}
A module $V$ with an additive dimension is called \emph{dim-connected} (for dimension-connected), or simply \emph{a dc-module}, if every proper submodule $W<V$ in $\cU$ satisfies $\dim W < \dim V$.
\end{definition*}

\begin{examples}\leavevmode
\begin{itemize}
\item
The only dc-module of dimension $0$ is $\{0\}$.
\item
Every finite-dimensional vector space over $\bF_p$ is dim-connected with respect to the linear dimension.
\item
If the dimension function satisfies `$\dim A = 0 \iff A$ is finite', then $V$ is dim-connected if and only if $V$ is connected in the usual model-theoretic sense of having no proper subobjects of finite index. But this is not so in general as we assume neither implication.
%
\end{itemize}
\end{examples}

\begin{remarks}\leavevmode
\begin{itemize}
\item
If $V$ contains a dc-submodule $V^1 \leq V$ with $\dim V^1 = \dim V$, then $V^1$ is unique (hence invariant under $\cU$-automorphisms of $V$) and called 
the \emph{dc-component} of $V$.
\item
If $W \leq V$ and both have dc-components, then $W^1 \leq V^1$.
Indeed, additivity implies $\dim ((W^1 + V^1)/V^1) = 0$, so $\dim (W^1/W^1\cap V^1) = 0$ also. Then dim-connectedness of $W^1$ forces $W^1 = W^1\cap V^1$.
\item
If $\cU$ satisfies the descending chain condition on objects, then dc-components exist. The converse is false ($0$-dimensional $\bZ$).
\item
In the universe of divisible abelian $p$-groups of finite Pr\"ufer rank, the dc-component of $V$ is the largest subtorus entirely contained in $V$.
\item
The dc-component of a group of finite Morley rank or an $o$-minimal group is its connected component.
\item
In general, even if dc-components exist, the index $[V:V^1]$ need not be finite ($0$-dimensional $V$).
\end{itemize}
\end{remarks}

We now highlight various operations preserving dim-connectedness.

\begin{cproperties}
Let $(\cU, \dim)$ be a modular universe with an additive dimension, from which we take modules.
\begin{enumerate}[label=(\roman*)]
\item
If $V$ is dim-connected and $f\colon V_1 \to V_2$ is compatible, then $\im f$ is dim-connected.
\item
If $V_1$ and $V_2$ are dim-connected, then so is $V_1 \times V_2$.
\item\label{cproperties.SumOfConnected}
If $V_1, V_2 \leq V$ and $V_1$ and $V_2$ are dim-connected, then so is $V_1+V_2$.
\end{enumerate}
\end{cproperties}
\begin{proof}\leavevmode
\begin{enumerate}[label=(\roman*)]
\item
Let $W_2 \leq \im f$ be a submodule of maximal dimension. Set $W_1 = f^{-1} (W_2) \geq \ker f$, and let $g\colon W_1 \to W_2$ be the restriction of $f$. Then $\dim W_1 = \dim \ker g + \dim \im g = \dim \ker f + \dim \im f = \dim V_1$. By dim-connectedness, $W_1 = V_1$, so $W_2 = \im f$.
\item
Suppose $Z \leq V_1 \times V_2$ is a submodule of maximal dimension. Let $\pi_i: V \rightarrow V_i$ be the projections and $\pi^Z_i$ the restrictions to $Z$.
Notice that $\ker \pi_i$ is compatibly isomorphic with $V_j$, hence dim-connected of the same dimension; of course $\ker \pi_i^Z \leq \ker \pi_i$.
Then $\dim V_1 + \dim V_2 = \dim (V_1 \times V_2) = \dim Z = \dim \im \pi^Z_i + \dim \ker \pi^Z_i \le \dim \im \pi_i + \dim \ker \pi_i = \dim V_1 + \dim V_2$. By dim-connectedness, one finds $\ker \pi^Z_1 = \{0\}\times V_2 \leq Z$ and $V_1 \times \{0\}\leq Z$, so $Z \geq (V_1 \times \{0\})\cdot(\{0\}\times V_2) = V_1\times V_2$.
%
\item
Apply the first two points to the sum map $V_1 \times V_2 \to V_1 + V_2$.
\qedhere
\end{enumerate}
\end{proof}

\subsection{The characteristic of a module}\label{s:characteristic}



\begin{definition*}
Let $V$ be a module with an additive dimension.
We say that:
\begin{itemize}
\item
$V$ has \emph{characteristic $p$}, for $p$ a prime, if it is of exponent $p$;
\item
$V$ has \emph{characteristic $0$} if it is divisible.
\end{itemize}
\end{definition*}

\begin{remarks}\leavevmode
\begin{itemize}
\item
A module need not have a well-defined characteristic: consider $\bZ/6\bZ$.
\item
Modules of characteristic $0$ may well contain torsion: consider $C_{p^\infty}$.
\end{itemize}
\end{remarks}

\begin{chlemma}
Let $V$ be a module with an additive dimension. If $V$ is dim-connected, then $V$ has a finite-length, $\Aut_\cU(V)$-invariant, dim-connected composition series $0 = V_0 < \dots < V_n = V$ in $\cU$ with $n \leq \dim V$ and each factor either of prime exponent or divisible.
\end{chlemma}
\begin{proof}
We proceed by induction on $\dim V$; the only dc-module of dimension $0$ is $\{0\}$. Let $p$ be a prime; consider the multiplication by $p$ morphism. The image $pV \leq V$ is dim-connected, so either $pV < V$ and we apply induction, or $p V = V$. In the latter case, $V$ is $p$-divisible, and we resume with another prime.
\end{proof}

\begin{examples}\leavevmode
\begin{itemize}
\item
Not all interesting modules have a characteristic. Any embedding $\Alt(4) \hookrightarrow \GL_2(\bZ/4\bZ)$ makes $V = (\bZ/4\bZ)^2$ a faithful, $2$-dimensional $\Alt(4)$-module of exponent $4$. It is not minimal, but $\Alt(4)$ is faithful on neither $2V$ nor $V/2V$.
(This of course relates to solubility; one should always be careful with $\rstd(4, 4)$.)
\item
Let $I$ be infinite and $T = \bigoplus_I C_{p^\infty}$ as a pure group. We suspect there is a non-trivial additive dimension on  $\HSP_{\rm def}(T)$; however Morley rank is infinite. Hence our setting seems to allow for tori of infinite Pr\"ufer rank.
%
\end{itemize}
\end{examples}

\begin{divproperties}
If $V$ is a dc-module with an additive dimension and $p$ is a prime, then $V$ is $p$-divisible iff $\Omega_p(V) := \{v \in V: pv = 0\}$ has dimension $0$. In particular, if $W \leq V$ is a dc-submodule and $V$ has a characteristic, then $V$ and $W$ have the same characteristic.
\end{divproperties}
\begin{proof}
The multiplication by $p$ morphism has kernel $\Omega_p(V)$; since $V$ is dim-connected, the map is onto if and only if the kernel has dimension $0$. When restricted to $W$, the kernel is $\Omega_p(W) = W \cap \Omega_p(V)$.
\end{proof}

\begin{remark}
In the case where $V = (C_{p^\infty})^n$ (which has characteristic $0$) with $p > 2$ a prime, it is well-known (for instance \cite{Dprank}) that the restriction morphism:
\[\rho : \Aut(V) \to \Aut(\Omega_p(V))\]
kills no element \emph{of finite order}. (There is a kernel $\{\pm 1\}^n$ if $p = 2$.)
We could however not make profit of this remark in our present work.
\end{remark}

\subsection{Dc-irreducibility and classes \texorpdfstring{$\Mod(G, d, q)$}{Mod(G,d,q)}}\label{s:dcirreducible}

\begin{definition*}
A dim-connected $G$-module $V$ is \emph{dc-irreducible} (as a $G$-module), if it has no non-trivial, proper, dim-connected $G$-submodule.
\end{definition*}

\begin{notation}
$\Mod(G, d)$ is the class of all dim-connected $G$-modules of dimension exactly $d$. (They are \emph{not} required to all live in a common universe; here one browses through universes and dimension functions.)
Subclasses $\Mod(G, d, q)$ specify the characteristic.
\end{notation}

By the \hyperlink{h:chlemma}{Characteristic Lemma} (\S~\ref{s:characteristic}), dc-irreducible modules always have a characteristic.

\subsection{Coprimality results}\label{s:coprimality}

We now let finite cyclic groups act on our modules; as one imagines, the characteristic plays a crucial role. Bear in mind that we do not assume the existence of dc-components for submodules (see remarks in \S~\ref{s:dc}); this explains why we avoid centralisers and prefer to work in terms of $\ad$ and $\tr$.

\begin{notation}
Let $V$ be a $\generated{g}$-module, where $g$ has order $p$.
\begin{itemize}
\item
Let $\ad_g = 1-g$ and $\tr_g = 1 + g + \dots + g^{p-1}$.
\item
Let $B_g = \im \ad_g$ and $C_g = \im \tr_g$.
\end{itemize}
\end{notation}

When $g$ acts compatibly on $V$ (in some universe), the \hyperlink{h:upropertiesis}{Universe Properties} ensure that $\ad_g$ and $\tr_g$ are compatible endomorphisms, so $B_g$ and $C_g$ are submodules. Be careful that $C_g$ does \emph{not} stand for the centraliser (though always $C_g\le C_V(g)$); for instance, if $g$ is an involution acting in characteristic $2$, then $\tr_g = \ad_g$ and $C_g = B_g$. However, in characteristic not $p$, pathologies are confined as shown below.

\begin{clemma}
Let $p$ be a prime and $V$ be a $p$-divisible, dim-connected $\generated{g}$-module with an additive dimension, where $g$ has order $p$.
Then $V = B_g + C_g$ and $\dim (B_g \cap C_g) = 0$.
\end{clemma}
\begin{proof}
As $B_g$ and $C_g$ are images under compatible endomorphisms of $V$, they are dc-submodules. Notice how $\ad_g \circ \tr_g = \tr_g \circ \ad_g = 1 - g^p = 0$ in $\End(V)$, so $B_g \leq \ker \tr_g$ and $C_g \leq \ker \ad_g$. However one sees $\ker \ad_g \cap \ker \tr_g \leq \Omega_p(V)$, which is $0$-dimensional by the \hyperlink{h:dpropertiesis}{Divisibility Properties}.
Thus,
\[\dim C_g = \dim \im \tr_g = \dim V - \dim \ker \tr_g \geq \dim \ker \ad_g \geq \dim C_g,\]
so equality holds. Then:
\[\dim (B_g + C_g) = \dim B_g + \dim C_g = \dim \im \ad_g + \dim \ker \ad_g = \dim V,\]
and $V = B_g + C_g$ by dim-connectedness.
\end{proof}

\begin{remarks}\leavevmode
\begin{itemize}
\item
The lemma proves that $B_g \cap C_g \le \Omega_p(V)$, so if $V$ has characteristic a prime different from $p$, then $B_g \cap C_g = 0$. But in characteristic $0$, $B_g \cap C_g$ need not be trivial, nor even finite.

 For instance let $T_1, T_2 \simeq C_{2^\infty}$ with central involutions $i_1, i_2$, and $\alpha$ an involution inverting $T_1$ while centralising $T_2$. Let $S = (T_1 \oplus T_2)/\generated{i_1 i_2}$. Notice how $B_\alpha = (1 - \alpha) S = \overline{T_1}$ and $C_\alpha = (1 + \alpha) S = \overline{T_2}$ intersect. Finally take an infinite direct sum of copies of $S$ and say it has dimension $1$.
\item
The proof also shows $C_g \leq \ker \ad_g = C_V(g)$ have the same dimension: hence $C_V(g)$ has a dc-component, which equals $C_g$. (Acting on a $p$-torus of infinite Pr\"ufer rank, one can produce examples with $[C_g: C_V(g)] = \infty$.) However, without assuming $p$-divisibility, $C_V(g)$ need no longer have a dc-component.
%
%
\end{itemize}
\end{remarks}

The conclusion of the \hyperlink{h:clemma}{Coprimality Lemma} is a `quasi-direct' decomposition for $V$. (This terminology has other meanings in the literature.)

\begin{definition*}
If $A_1,\ldots,A_n$ are submodules of a module $V$, the sum $\sum A_i$ is said to be \emph{quasi-direct} if $\dim \sum A_i = \sum \dim A_i$, in which case we write $\sum A_i = A_1\qoplus \cdots\qoplus A_n$.
\end{definition*}

\begin{remark}
Note that $\sum_1^n A_i$ is quasi-direct if and only if \emph{both}: $\sum_1^{n-1} A_i$ is quasi-direct, and $\dim \left(A_n \cap \sum_1^{n-1} A_i\right) = 0$.
\end{remark}

The following lemma highlights an important application of the \hyperlink{h:clemma}{Coprimality Lemma}---it will be used often in the sequel.

\begin{wlemma}
Let $E$ be a finite elementary abelian $2$-group and $V$ a $2$-divisible, dim-connected $E$-module with an additive dimension.
Then $V$ decomposes into a quasi-direct sum of dim-connected weight submodules $V_\lambda$ where $\lambda:V\rightarrow \{\pm1\}$ is a group morphism and each $e\in E$ acts on $V_\lambda$ as $\lambda(e)$. Each $\lambda$ is called a \emph{weight} of $E$ and $V_\lambda$ the corresponding \emph{(dim-connected) weight space}.
\end{wlemma}
\begin{proof}
Write $E = E_0 \oplus \langle e\rangle$. By the \hyperlink{h:clemma}{Coprimality Lemma}, $V = B_e \qoplus C_e$ with $e$ inverting the first factor and centralizing the latter. Applying induction to the action of $E_0$ on each of $B_e$ and $C_e$ then yields the desired result.
%
\end{proof}

\begin{remark}
If $V$ is a $2$-divisible, dim-connected $G$-module and $K\le G$ is a Klein four-subgroup whose nontrivial elements are conjugate in $G$, then the spaces attached to non-trivial weights have constant dimension $\ell$, and $\dim V = \dim C_V(K) + 3\ell$. This will be used repeatedly.
\end{remark}

\section{The proof}\label{S:proof}

The proof of the \hyperlink{h:theorem}{Theorem} will be assembled from three components: 
the \hyperlink{h:rlemma}{Recognition Lemma} (\S~\ref{s:rlemma}) provides an elementary geometric condition sufficient to identify the standard $\Sym(n)$-module (and its quotients in relevant characteristics);
the \hyperlink{h:elemma}{Extension Lemma} (\S~\ref{s:elemma}) uses an analogous geometric condition to  identify when an $\Alt(n)$-module extends to $\Sym(n)$ in such a way that \hyperlink{h:rlemma}{Recognition} applies; and 
the \hyperlink{h:glemma}{Geometrisation Lemma}  (\S~\ref{s:glemma}) details how $\Alt(n)$-modules of low dimension naturally possess the geometric condition needed to invoke \hyperlink{h:elemma}{Extension} (and thus \hyperlink{h:rlemma}{Recognition}).
Notably, the first two of these results are quite general with hypotheses only on the existence of a dimension function, and no restriction on its values.
\begin{notation}
\leavevmode
\begin{enumerate}[label=\arabic*.]
\item
\textit{Elements.} We reserve $i, j, k, \dots$ for elements of $\{1, \dots, n\}$. Permutations are typically denoted by lower-case greek letters, reserving:
\begin{itemize}
\item
$\tau, \tau'$ for transpositions;
\item
$\alpha, \beta$ for bi-transpositions;
\item
$\gamma, \delta$ for $3$-cycles.
\end{itemize}
We indicate that permutations $\sigma_1$ and $\sigma_2$ have disjoint supports $|\sigma_1|$ and $|\sigma_2|$ by writing $\sigma_1\perp\sigma_2$.
\item
\textit{Subgroups.} 
\begin{itemize}
\item
We avoid stabiliser notation $G_{i, j}$ and $G_{\{i, j\}}$. Instead, if $S = \Sym(n)$ and $I \subseteq \{1, \dots, n\}$, we let $S_I \leq S$ be the subgroup of permutations with support contained in $I$. For $\sigma \in S$, we let $S_\sigma = S_{|\sigma|}$ and $S_{\sigma^\perp} = S_{|\sigma|^c}$. Likewise in $\Alt(n)$.
\item
We often consider subgroups of $A = \Alt(n)$ isomorphic to $\Sym(k)$. Typically, for $|I|\leq n-2$ and symbols $k, \ell \notin I$, we let:
\begin{align*}
\Sigma_I^{(k\ell)} & = A_I \sqcup \left((S_I \setminus A_I) \cdot (k \ell)\right)\\
& = 
\left\{\begin{array}{ccc}\sigma & \text{ if } & \varepsilon(\sigma) = 1\\ \sigma(k \ell) & \text{ if } & \varepsilon(\sigma) = -1\end{array}: \sigma \in \Sym(I)\right\}
,\end{align*}
a subgroup of $\Alt(n)$ isomorphic to $\Sym(I)$.
When there is no ambiguity we simply write $\Sigma_I$.
\item
We use $K$ for Klein four-groups of bitranspositions, with $K_{ijk\ell}$ the Klein four-group having support $\{i,j,k,\ell\}$.
\end{itemize}
\end{enumerate}
\end{notation}

Also, recall:
\begin{itemize}
\item
from \S~\ref{s:dcirreducible}, that $\Mod(G, d, q)$ stands for the class of \emph{dim-connected} $G$-modules of dimension $d$ and characteristic $q$;
\item
from \S~\ref{s:dcirreducible} as well, the notion of \emph{dc-irreducibility} meaning irreducibility in the class of dim-connected $G$-modules;
\item
from \S~\ref{s:coprimality}, that for $V$ a $G$-module and $g \in G$, we define $B_g := [g, V] = \im \ad_g$ where $\ad_g = 1 - g \in \End(V)$.
\end{itemize}

\subsection{Recognising the standard module}\label{s:rlemma}

The \hyperlink{h:rlemma}{Recognition Lemma} constructs a natural covering module under assumptions of an elementary geometric nature. In prime characteristic (or in the torsion-free case), the kernel is known, and the isomorphism type fully determined.

Recall from \S~\ref{s:results} that $\ustd(n, \bZ) = \generated{f_1, \dots, f_{n-1}}$ is the $\Sym(n)$-submodule of the permutation module $\perm(n, \bZ) = \generated{e_1, \dots, e_n}$ generated by $f_i = e_i - e_n$. For any abelian group $L$, $\ustd(n, L) := \ustd(n, \bZ) \otimes_\bZ L$, and $\rstd(n, L) = \ustd(n, L)/C_{\ustd(n,L)}(\Sym(n))$. 

\begin{rlemma}
Let $n\ge 1$, $S := \Sym(n)$, and $V \in \Mod (S, d, q)$ be faithful and dc-irreducible. 
Suppose that for any transposition $\tau$, one has $[S'_{\tau^\perp}, B_\tau] = 0$.

Then for some abelian group $L$ and arrow $\varphi$ in $\cU$, there is a surjective morphism $\varphi\colon \ustd(n, L) \twoheadrightarrow V$ of $S$-modules.
Moreover:
\begin{itemize}
\item
if $0 \neq q \mid n$, then $\ker \varphi = C_{\ustd(n, L)}(S)$ and $V \simeq \rstd(n, L)$;
\item
if $0\neq q \nmid n$, then $\ker \varphi = 0$ and $V \simeq \rstd(n, L) \simeq \ustd(n, L)$;
\item
if $q = 0$, then $\ker \varphi = \generated{c}\otimes K$ where $c = \sum_{i = 1}^{n-1} f_i$ and $K \leq \Omega_n(L)$ is $0$-dimensional.
\end{itemize}
\end{rlemma}

\begin{remark}
In the first two cases, $V$ is completely reducible into a direct sum of isotypical summands $\rstd(n, \bF_q)$; when $V$ is torsion-free, the same is true with summands of the form $\ustd(n, \bQ)$. However, actions on tori could give rise to non-trivial quotients (in characteristic $0$).
\end{remark}

\begin{proof}
The case of $n \le 2$ is clear, so we suppose $n \geq 3$.
 By dc-irreducibility, $V = [S,V] = \sum_{i = 1}^{n-1} B_{(in)}$; this will be used several times below.

\begin{claim}[local equations]\label{l:recognition:cl:local}
Let $i \neq j$. Then $\ad_{(ij)} = 1 - (ij)$ acts as:
$\left\{\begin{array}{ll} 2 & \text{on $B_{(ij)}$}\\
(jk) & \text{on $B_{(ik)}$ for $k \notin \{i, j\}$}\\
0 & \text{on $B_{(k\ell)}$ for  $\{k, \ell\}\perp\{i, j\}$.}\end{array}\right.$
\end{claim}
\begin{proofclaim}
This is obvious on $B_{(ij)}$; of course $2 = 0$ is a possibility.

We turn to the middle case. Let $A:= S'$. First, note that $B_{(ij)}\cap B_{(jk)}$ is $S$-invariant. Indeed, it is  $S_{\{i,j,k\}}$-invariant as it is inverted by both $(ij)$ and $(jk)$, and by our main assumption, it is also centralised by $\langle A_{(ij)^\perp},A_{(jk)^\perp} \rangle = A_{j^\perp}$,  implying $S$-invariance.
Now let $a \in B_{(ik)}$, and write:
\begin{align*}
\underbrace{[1 - (ij)] a}_{\in B_{(ij)}} - \underbrace{(jk) a}_{\in B_{(ij)}} & = [1-(ij) - (jk)] a = - [1 - (ij) - (jk)] (ik) a\\
& = -\underbrace{[(ik) - (jk)(ik)] a}_{\in B_{(jk)}} + \underbrace{(ikj)a}_{\in B_{(jk)}}.
\end{align*}
Thus, $X := [1 - (ij) - (jk)]B_{(ik)}$ is a dim-connected subgroup of $B_{(ij)} \cap B_{(jk)}$. As noted above, $B_{(ij)} \cap B_{(jk)}$ is $S$-invariant, so $\sum_{\sigma\in S} \sigma X$ is a dim-connected and $S$-invariant subgroup of  $B_{(ij)} \cap B_{(jk)} < V$, hence trivial by dc-irreducibility. Hence $X = 0$, as desired.

It remains to verify the third equation, for which we may assume $n\geq 4$. As $[A_{\tau^\perp}, B_\tau] = 0$, each $\tau'\perp \tau$ gives rise to the same $B_\tau^+ := \tr_{\tau'} (B_\tau)$, which is centralised by $S_{\tau^\perp}$. We aim to show $B_\tau^+ = B_\tau$.

For $s \in S$, $s(B_{(ij)}^+) = B_{(s(i)s(j))}^+$, so
$V^+ := \sum_{i \neq j} B_{(ij)}^+$ is  $S$-invariant. Notice how $\ad_{(ij)}B_{(k\ell)}^+ = \ad_{(ij)}\tr_{(ij)}(B_{(k\ell)}) = 0$, and (by the second local equation)
\begin{align*}
\ad_{(ij)} (B_{(ik)}^+) & = (jk)B_{(ik)}^+ = (jk)\tr_{(j\ell)}(B_{(ik)}) = (jk)\tr_{(j\ell)}(jk)B_{(ij)}\\
&  = \tr_{(k\ell)}B_{(ij)} = B_{(ij)}^+.
\end{align*}
Thus, $\ad_{(ij)} V^+ = B_{(ij)}^+$.
By dc-irreducibility, $V^+$ equals  $0$ or $V$, so either $B_{(ij)}^+ = \ad_{(ij)} V^+ = 0$ or $B_{(ij)}^+ = \ad_{(ij)} V^+ = \ad_{(ij)} V = B_{(ij)}$. The latter  is our goal, so it remains to consider  $B_{(ij)}^+ = 0$. 

Assume  $B_{(ij)}^+ = 0$. If $q = 2$, we have the desired result since  $0 = B_{(ij)}^+ = \tr_{(k\ell)}(B_{(ij)})  = \ad_{(k\ell)}(B_{(ij)})$. If $q \neq 2$, then \hyperlink{h:clemma}{Coprimality} implies that $B_{(ij)} = B_{(k\ell)}$; conjugating, this quickly contradicts dc-irreducibility when $n\ge 5$. If $n=4$, then $V = B_{(12)} + B_{(23)} + B_{(34)}$; since$B_{(12)} = B_{(34)}$, one has $V = B_{(12)} + B_{(23)}$.  
Moreover, $V = B_{(12)} + B_{(23)}$ is a  $K$-invariant decomposition (for $K$ the Klein four-group) with the first factor centralised by $(12)(34)$ and the second by $(23)(14)$. A weight space decomposition of $V$ for $K$ must then result in the trivial weight space having positive dimension (by the \hyperlink{h:wlemma}{Weight Lemma}), against dc-irreducibility and faithfulness.
\end{proofclaim}


We construct a covering $S$-module as follows: 
\begin{itemize}
\item
let $L := B_{(1n)}$ as a trivial $S$-module;
\item
let $\hat{V} := \ustd(n, L)$ as an $S$-module, and define $f_i = e_i - e_n$ as usual;
\item
let $\varphi: \hat{V} \to V$ be the additive map such that $\varphi(f_i \otimes \ell) = (1i) \cdot \ell$, where $(11)$ is interpreted as the identity. 
\end{itemize}
Every element of $\hat{V}$ has a unique decomposition as $\sum_{i = 1}^{n-1} f_i\otimes \ell_i$ for $\ell_i\in L$, so $\varphi$ is indeed well-defined and  additive. Note that in the image of $\varphi$, we use the non-trivial action of $S$ on $B_{(1n)} \leq V$.

By the Universe Properties of \S~\ref{s:universe}, $\perm(n, L)$ is an $S$-module in $\cU$, so each $f_i \otimes L$, being the image of $\perm(n, L)$ under $\ad_{(in)}$, is also in $\cU$. Taking the sum, we find that $\hat{V}$ is an $S$-module in $\cU$. Now let $\varphi_i: f_i \otimes L\rightarrow V$ be $(1i)\circ\iota\circ\pi_i$ where $\iota:L\rightarrow V$ is the inclusion map and $\pi_i$ is the restriction to $f_i \otimes L$ of the $i^\text{th}$-projection of $\perm(n, L) = L^n$. Thus, each $\varphi_i$ is compatible so their common extension to $\hat{V}$ (which is $\varphi$) is as well.


\begin{claim}[covering]\label{l:recognition:cl:varphi}
$\varphi$ is a surjective morphism of $S$-modules.
\end{claim}
\begin{proofclaim}
Notice that $\varphi(f_i \otimes \ell) \in (1i) \cdot B_{(1n)} = B_{(in)}$. 
Therefore $\im \varphi \geq \sum_{i = 1}^{n-1} B_{(in)} = V$, so $\varphi$ is surjective. It remains to prove covariance. We use Claim~\ref{l:recognition:cl:local} freely.

Since $\{(jn): 1 \leq j \leq n-1\}$ generates $S$ and  $\{f_i \otimes \ell:1 \leq i \leq n-1,\ell \in L\}$ generates $\hat{V}$, it is enough to treat the  basic cases. If $i = j$, then
\[\varphi((in)\cdot (f_i \otimes \ell)) 
= - \varphi(f_i \otimes \ell) = (in) \cdot \varphi(f_i \otimes \ell).\]
If $i \neq j$, then
\[\varphi((jn)\cdot (f_i \otimes \ell)) = \varphi((f_i - f_j)\otimes \ell) = (1i) \ell - (1j)\ell.\]
To compute $(jn)\cdot\varphi(f_i \otimes \ell)$, we consider separately $j=1$ or not. If $j=1$,
\[(1n)(1i)\ell = (1i)(in)\ell = (1i)[(in) - 1 + 1]\ell = - (1i)^2 \ell + (1i)\ell,\]
completing this case. And if $j \neq 1$, 
\[(jn)(1i)\ell = (1i)(jn) \ell = (1i)[(jn)-1+1]\ell = -(1i)(1j)\ell + (1i)\ell,\]
which establishes this case since $(1j)\ell \in B_{(jn)}$ is centralised by $(1i)$.
\end{proofclaim}


\begin{claim}[kernel control]\label{l:recognition:cl:kervarphi}\leavevmode
\begin{itemize}
\item
If $0 \neq q \mid n$, then $\ker \varphi = C_{\hat{V}}(S)$ and $V \simeq \rstd(n, L)$;
\item
if $0\neq q \nmid n$, then $\ker \varphi = 0$ and $V \simeq \rstd(n, L) \simeq \ustd(n, L)$;
\item
if $q = 0$, then $\ker \varphi = \generated{c}\otimes K$ where $c = \sum_{i = 1}^{n-1} f_i$ and $K \leq \Omega_n(L)$ is $0$-dimensional.
\end{itemize}
\end{claim}
\begin{proofclaim}
Let $c = \sum_{i = 1}^{n-1} f_i \in \ustd(n, \bZ)$. We first contend that:
\[\varphi^{-1}(C_V(S)) = C_{\hat{V}}(S) = \{c\otimes \ell: \ell \in L: n\ell  = 0\} = \generated{c}\underset{\bZ}{\otimes} \Omega_n(L) \simeq \Omega_n(L).\]
Notice at once that $c$ generates the centraliser of $\Sym(1, \dots, n-1)$, in symbols $C_{\ustd(n, \bZ)}(S_{n^\perp}) = \generated{c}$; the same holds in $\ustd(n, L)$, viz.~$C_{\hat{V}}(S_{n^\perp}) = \generated{c} \otimes L$.
Also, in $\ustd(n, \bZ)$ one sees $(1n)c = c - n f_1$, so  $C_{\hat{V}}(S) = \{c\otimes \ell: n \ell = 0\} = \generated{c} \otimes \Omega_n(L) \simeq \bZ \otimes_\bZ \Omega_n(L) \simeq \Omega_n(L)$. It remains to prove $\varphi^{-1}(C_V(S)) = C_{\hat{V}}(S)$. The latter is clearly contained in the former; we now show that $v \in C_V(S)$ is in the image of $C_{\hat{V}}(S)$. As noted after Claim~\ref{l:recognition:cl:local}, $v = \sum_{i = 1}^{n-1}v_i$ with $v_i \in B_{(in)}$. Applying $\ad_{(1i)}$ and using the local computations, we see that $(in)v_1 + (1n)v_i = 0$, implying that $v_i = (1i)v_1$. Thus, $v = \sum_{i = 1}^{n-1}(1i)v_1 = \varphi(c\otimes v_1)$.

We may now finish the proof. Clearly $\ker \varphi \leq \varphi^{-1} (C_V(S)) = C_{\hat{V}}(S)$.
\begin{itemize}
\item
First suppose $q \neq 0$ and $q \nmid n$. Then $\Omega_n(L)$ and $C_{\hat{V}}(S) \ge \ker \varphi$ are trivial, so $V \simeq \hat{V} = \ustd(n, L)$. (The same holds if $V$ is torsion-free.)
\item
Next suppose $q \neq 0$ and $q | n$. Then $L = \Omega_n(L)$ so $C_{\hat{V}}(S) = \generated{c}\otimes L \simeq L$; the image $\varphi(C_{\hat{V}}(S)) = C_V(S)$ is now a quotient module of $L$, hence dim-connected. By dc-irreducibility of $V$, $C_V(S)$ is trivial, so $C_{\hat{V}}(S) = \ker \varphi$. Hence $V \simeq \hat{V}/\ker \varphi = \rstd(n, L)$.
(The last is seen by recalling that $q$ is prime and $\rstd(n,\bF_q)$ is irreducible.)
\item
Finally suppose $q = 0$. Recall that $\Omega_n(V)$ is $0$-dimensional by the \hyperlink{h:divproperties}{Divisibility Properties}.
Subgroups of $\generated{c} \otimes \Omega_n(L) \simeq \Omega_n(L)$ are of the form $\generated{c}\otimes K$ for $K \leq \Omega_n(L)$, and $\ker \varphi$ is one such.\qedhere
\end{itemize}
\end{proofclaim}
This completes the proof of the \hyperlink{h:rlemma}{Recognition Lemma}.
\end{proof}\setcounter{claim}{0}

On can also rephrase Claim~\ref{l:recognition:cl:varphi} of the \hyperlink{h:rlemma}{Recognition Lemma} as follows, with no reference to dimensionality.

\begin{corollary*}
Let $V$ be an abelian group equipped with an irreducible, faithful action of $\Sym(n)$. Suppose that for any two distinct transpositions $\tau, \tau'$ one has $\sum_{g \in \generated{\tau, \tau'}} \varepsilon(g) g = 0$.
Then $V$ is a homomorphic image of some $\ustd(n, L)$.
\end{corollary*}
\begin{proof}
The `integration formula' implies $[1 - (ij)][1-(k\ell)] = 0$ (when $\tau\perp \tau'$) and $[1 - (ij)] [1 - (ik)] = (jk)[1 - (ik)]$ (when $\tau\not\perp \tau'$) in $\End(V)$. Hence $V$ satisfies the conclusion of Claim~\ref{l:recognition:cl:local}, which is enough to produce a covering map $\varphi\colon \ustd(n, \bZ) \otimes_\bZ B_{(1n)}\to V$.
\end{proof}

\subsection{Extending an \texorpdfstring{$\Alt(n)$}{Alt(n)}-module to \texorpdfstring{$\Sym(n)$}{Sym(n)}}\label{s:elemma}

We now turn to $\Alt(n)$-modules, giving a geometric condition (analogous to that for the \hyperlink{h:rlemma}{Recognition Lemma}) under which an $\Alt(n)$-module extends to a $\Sym(n)$-module subject to Recognition. But do note that the two lemmas, Recognition and Extension, are independent.

\begin{elemma}[cf.~{\cite{WaA76,WaA77}}]
Let $n\ge 7$, $A:= \Alt(n)$, and $V \in \Mod(A, d, q)$ be faithful and dc-irreducible.
Suppose that for any bi-transposition $\alpha$, one has $[A_{\alpha^\perp}, B_\alpha] = 0$.
Then:
\begin{itemize}
\item
if $q = 2$ there is a unique compatible action of $\Sym(n)$ extending the $\Alt(n)$-structure;
\item
if $q \neq 2$ there are exactly two such, obtained from each other by tensoring with the signature.
\end{itemize}
Moreover, up to tensoring with the signature, the extension satisfies that 
for any transposition $\tau \in S:=\Sym(n)$, one has $[S_{\tau^\perp}, B_\tau] = 0$.
\end{elemma}

\begin{remark}
If $n \geq 8$ the main assumption is equivalent to: for any $3$-cycle $\gamma$, one has $[A_{\gamma^\perp}, B_\gamma] = 0$.
The `$\gamma$'-version is however stronger if $n = 7$.
\end{remark}

\begin{proof}
The bulk of the proof is devoted to existence; uniqueness will result afterwards.
We aim to extend the action of $\Alt(n)$ to $\Sym(n)$; to that end, we first identify what should be $B_{(ij)} = [(ij),V]$, which (up to tensoring with the sign representation) should be thought of as a line that we will call  $L_{(ij)}$. 
In the standard module, $L_{(ij)}$ can be computed as $B_\alpha \cap B_\gamma$ using any bi-transposition $\alpha$ that swaps $i$ and $j$ and any $3$-cycle $\gamma$ satisfying $|\alpha|\cap |\gamma| = \{i, j\}$. (There are certainly other ways to isolate $B_{(ij)}$ such as by intersecting $B_\alpha$ and $B_\beta$ for $\alpha$ and $\beta$ distinct bi-transpositions that both swap $i$ and $j$; this was in fact our original point of view.)

\begin{claim}\label{elemma:cl:L}
Let $i \neq j$ be given. For distinct $a,b,k\notin \{i,j\}$, \[L_{(ij)}:= \im \ad_{(ij)(ab)} \circ \ad_{(ijk)} = (B_{(ij)(ab)} \cap B_{(ijk)})^1\] is nontrivial and independent of the choice of $a,b,k$. Also, $[A_{\{i, j\}^\perp},L_{(ij)}]=0$.
\end{claim}
\begin{proofclaim}
Consider $L_{(ij)}:= \ad_{(ij)(ab)} B_{(ijk)}$ for distinct $a,b,k\notin \{i,j\}$; set $\alpha := (ij)(ab)$ and $\gamma := (ijk)$.

We first prove centralisation by $A_{\{i, j\}^\perp}$ and independence from $a, b, k$. Note that the group $(A_{\gamma^\perp})'$ is generated by its bi-transpositions $\beta$ (when $n = 7$, this fails of $A_{\gamma^\perp}$ itself), which all commute with $\gamma$ and satisfy:
\[\ad_\beta (B_\gamma) = \ad_\beta \circ \ad_\gamma (V) = \ad_\gamma \circ \ad_\beta (V) = \ad_\gamma(B_\beta) \leq [A_{\beta^\perp}, B_\beta] = 0,\]
implying that $[(A_{\gamma^\perp})', B_\gamma] = 0$. Also notice that $\alpha$ inverts $\gamma$. In particular $\ad_\alpha$ leaves $\im \ad_\gamma$ invariant; hence $L_{(ij)}\le B_{\alpha} \cap B_{\gamma}$, so by assumption and what we just noted, $L_{(ij)}$ is centralised by $\langle A_{\alpha^\perp}, (A_{\gamma^\perp})'\rangle = A_{\{i,j\}^\perp}$ (even if $n=7$). And as $A_{\{i,j\}^\perp} \ge \Alt(5)$ is $3$-transitive off of $\{i,j\}$, we also find that $L_{(ij)}$ is independent of the choice of $a,b,k$. 

We now show $L_{(ij)}= (B_{\alpha} \cap B_{\gamma})^1$, which will follow readily from \hyperlink{h:clemma}{Coprimality} and dim-connectedness of $L_{(ij)} = \ad_\alpha \circ \ad_\gamma (V)$.
If $\charac V \neq 2$, then letting $\alpha$ act on $B_\gamma$ we find $B_{\gamma} = L_{(ij)} \qoplus \tr_{\alpha}(B_{\gamma})$. Now $B_{\alpha} \cap \tr_{\alpha}(B_{\gamma}) \le B_{\alpha} \cap C_{\alpha}$ has dimension $0$, so $\dim  (B_{\alpha} \cap B_{\gamma}) = \dim( B_{\alpha} \cap (L_{(ij)} \qoplus \tr_{\alpha}(B_{\gamma})) = \dim L_{(ij)} + \dim (B_{\alpha} \cap \tr_{\alpha}(B_{\gamma})) = \dim L_{(ij)}$.
And if $\charac V \neq 3$,  write $V = B_{\gamma} \qoplus C_{\gamma}$, an  $\alpha$-invariant decomposition. 
Applying $\ad_{\alpha}$, we find that $B_{\alpha} = L_{(ij)} + \ad_{\alpha}(C_{\gamma})$ with $B_{\gamma}\cap \ad_{\alpha}(C_{\gamma})$ being $0$-dimensional. As before, we find that $\dim(B_{\alpha} \cap B_{\gamma}) = \dim L_{(ij)}$.

It remains to show that $L_{(ij)}$ is nontrivial, which is equivalent to showing that $\alpha$ does not centralise $B_{\gamma}$. Suppose it does. Conjugating, $\alpha$ centralises $B_{(ij\ell)}$ for all $\ell\notin\{i,j,a,b\}$ and also $B_{(ab\ell)}$ for all $\ell\notin\{i,j,a,b\}$. Thus, $C_V(\alpha)$ contains both $[A_{\{a,b\}^\perp},V]$ and $[A_{\{ij\}^\perp},V]$, hence all of $[A,V]$. This  contradicts our assumptions of dc-irreducibility and faithfulness.
\end{proofclaim}

\begin{remark}
There is a counterexample to Claim~\ref{elemma:cl:L} when $n=6$. In the case of the adjoint representation of $\Alt(6)\simeq \PSL_2(\bF_9)$,  $[A_{\alpha^\perp},B_\alpha] = 0$ and $L_{(ij)}$ has positive dimension, but $L_{(ij)}$ is \emph{not} independent of the choice of $a,b,k$. 
\end{remark}

The next claim establishes various expected properties of the $L_{(ij)}$; recall that $L_{(ij)}$ is a proxy for $B_{(ij)}$.

\begin{claim}[Geometry of lines]\label{elemma:cl:LijProperties}\leavevmode
\begin{enumerate}
\item\label{i:linestotal}
$V = \sum_{i\neq j} L_{(ij)}$.
\item\label{i:disjointlines}
If $\{i, j\} \neq \{k, \ell\}$ are distinct pairs, then $L_{(ij)} \cap L_{(k\ell)} = 0$.
\item\label{i:localequations}
If $i, j, k, \ell,x$ are distinct symbols, then: 
\begin{itemize}
\item $\ad_{(ijk)}(L_{(ix)}) = L_{(ij)}$;
\item $\ad_{(ij)(k\ell)}(L_{(ix)}) = L_{(ij)}$.
\end{itemize}
\item\label{i:linegeometry} $L_{(ij)} \leq L_{(ik)} + L_{(jk)}$.
\end{enumerate}
\end{claim}
\begin{proofclaim}
Clearly $\sum_{i\neq j} L_{(ij)}$ is $A$-invariant so must be equal to $V$, establishing \eqref{i:linestotal}.

We now handle \eqref{i:disjointlines} and will prove that distinct lines are disjoint. We   
 first consider {disjoint} index sets; by conjugacy, we may take $i, j, k, \ell$ to be $1,2,3,4$. Using Claim~\ref{elemma:cl:L}, \[L_{(12)}\cap L_{(34)} \leq B_{(125)} \cap C_V((125)) \leq  \Omega_3(V),\] but on the other hand: \[L_{(12)}\cap L_{(34)} \leq B_{(12)(56)} \cap C_V((12)(56)) \leq  \Omega_2(V).\]
Thus, $L_{(12)} \cap L_{(34)} = 0$, so lines with \emph{disjoint} index sets are disjoint;
we turn to intersecting sets. Notice that $L_{(12)} \cap L_{(23)} \leq B_{(12)(45)} \cap B_{(23)(45)} \leq C_V((123))$. Hence:
\[L_{(12)} \cap L_{(23)} \leq B_{(123)} \cap C_V((123)) \leq \Omega_3(V).\]
On the other hand $L_{(12)} \cap L_{(23)} \leq B_{(12)(34)} \cap B_{(23)(14)} \leq C_V((13)(24))$, so:
\begin{align*}
L_{(12)} \cap L_{(23)} & \leq B_{(12)(34)} \cap \left[C_V((13)(24)) \cap C_V((123))\right] \\ & \leq B_{(12)(34)} \cap C_V((12)(34)) \leq \Omega_2(V).\end{align*}
Therefore $L_{(12)} \cap L_{(23)} = 0$; distinct lines are disjoint.


We now prove \eqref{i:localequations}. Here we take $i, j, k, x$ to be $1,2,3,4$. 
Let $\alpha = (12)(56)$, $\beta = (23)(56)$, and $\gamma = \alpha \beta = (123)$. We show $\ad_{\gamma}(L_{(14)}) = L_{(12)}$.
Observe:
\[\ad_{\gamma} = 1 - \alpha \beta = 1 - \beta + (1 - \alpha) \beta = \ad_\beta + \ad_\alpha \circ \beta. \qquad(\ast)\]
If $v \in L_{(14)}$, then $v \in C_V(\beta)$, so applying $(\ast)$ we find $\ad_{\gamma}(v) = \ad_\alpha(v) \in B_\alpha$. Thus $\ad_{\gamma}(L_{(14)}) \le (B_{\gamma} \cap B_{\alpha})^1 = L_{(12)}$, by Claim~\ref{elemma:cl:L}. And as $\gamma \cdot L_{(14)} = L_{(24)}$, we have $ L_{(14)}\cap C_V(\gamma) \le L_{(14)}\cap L_{(24)} = 0$ by \eqref{i:disjointlines}, 
 so $L_{(14)} \cap \ker \ad_{\gamma}$ is trivial. Thus,   $\dim \ad_{\gamma}(L_{(14)}) = \dim L_{(14)} = \dim L_{(12)}$, forcing $\ad_{\gamma}(L_{(14)}) = L_{(12)}$.  

For the second part of \eqref{i:localequations}, we keep $\alpha, \beta,\gamma$ as before and  show $\ad_{\beta} L_{(24)} = L_{(23)}$.  
Now, $\beta = \alpha\gamma$, so here $\ad_{\beta} = \ad_\gamma + \ad_\alpha \circ \gamma$. But $\ad_\alpha \circ \gamma (L_{(24)}) = \ad_\alpha (L_{(34)}) = 0$, so by our previous work $\ad_{\beta}(L_{(24)}) = \ad_\gamma(L_{(24)}) = L_{(23)}$.

For \eqref{i:linegeometry}, consider $H := \Sigma_{\{1,2,3\}}^{(45)} \simeq \Sym(3)$, and observe that  $[H,B_{(123)}] = \ad_{(13)(45)}(B_{(123)}) + \ad_{(23)(45)}(B_{(123)}) =L_{(13)} + L_{(23)}$. Of course,   $[H,B_{(123)}]$ also contains $\ad_{(12)(45)}(B_{(123)}) = L_{(12)}$.
%
\end{proofclaim}

For $I \subseteq \{1,\ldots,n\}$ with $|I| \ge 2$, define \[V_{I} := \sum_{(i,j)\subseteq I} L_{(ij)}. 
\]

\begin{claim}\label{elemma:cl:V}
Let $I,J \subseteq \{1,\ldots,n\}$ with $|I|,|J|\ge 2$. The following hold:
\begin{enumerate}
\item\label{i:VItotal} $V_{\{1, \dots, n\}} = V$;
\item\label{i:sumIJ} if $|I\cap J| \ge 1$, then $V_{I\cup J} = V_I + V_J$;
\item\label{i:LVI} if $|I| \le n-3$, then  $V_I \cap L_{(ab)}= 0$ for $a, b \notin I$;
\item\label{i:n-2lines} the sum $L_{(12)} + L_{(23)} + \dots + L_{(n-2, n-1)}$ is direct;
\item\label{i:AIperponVI}
$[A_{I^\perp}, V_I] = 0$;
\item\label{i:bracketAI}
$[A_I, V] = V_I$ provided $|I| \geq 3$.
\end{enumerate}
\end{claim}
\begin{proofclaim}
Part \eqref{i:VItotal} is merely Claim~\ref{elemma:cl:LijProperties}\eqref{i:linestotal}, and \eqref{i:sumIJ} follows readily from Claim~\ref{elemma:cl:LijProperties}\eqref{i:linegeometry}. Parts  \eqref{i:AIperponVI} and \eqref{i:bracketAI} are also fairly immediate. Indeed,  $[A_{I^\perp}, V_I] = \sum_{(i, j)\subseteq I} [A_{I^\perp}, L_{(ij)}] = 0$ by Claim~\ref{elemma:cl:L}, and if $|I| \geq 3$, then
\[[A_I, V] = \sum_{|\gamma| \subseteq I} B_\gamma = \sum_{(i, j, k) \subseteq I} L_{(ij)} + L_{(jk)} = \sum_{(i, j)\subseteq I} L_{(ij)}.\]

For \eqref{i:LVI}, choose distinct $a,b,k \notin I$; set  $X := L_{(ab)}\cap V_I$. Let $i \in I \setminus \{a, b, k\}$ and  $\gamma := (iak)$. Then using all of Claim~\ref{elemma:cl:LijProperties} we find $\ad_\gamma(X) \leq L_{(ak)} \cap L_{(ia)} = 0$, so $X \leq L_{(ab)} \cap \gamma(L_{(ab)}) = L_{(ab)} \cap L_{(kb)} = 0$. This establishes \eqref{i:LVI}, and  \eqref{i:n-2lines} now follows readily by induction.
\end{proofclaim}

For fixed $i\neq j$, let 
\[E_{i, j}:=\{\alpha \mid \text{$\alpha$ is a bi-transposition exchanging $i$ and $j$}\}.\]

The final ingredient we need to define the action of a transposition $\tau$ is the hyperplane representing $C_V(\tau)$, which we now define as $H_{\tau}$. (Regarding our definition below, recall that $C_\alpha = \tr_\alpha V$, and though always contained in $C_V(\alpha)$, it may be significantly smaller.)

\begin{claim}\label{elemma:cl:H}
Let $H_{(ij)} := \sum_{\alpha \in E_{i, j}} C_\alpha$. 
For distinct $k,a,b \notin \{i, j\}$, we have $V = V_{\{i,j,k\}} + H_{(ij)}$ with $V_{\{i,j,k\}} \cap H_{(ij)} \leq C_V((ij)(ab))$.
Further, $V_{\{i, j\}^\perp}\le H_{(ij)}$.
\end{claim}
\begin{proofclaim}
We first show $V_{\{i, j\}^\perp} \leq H_{(ij)}$. Let $(k\ell) \perp (ij)$. If $q = 2$, then since $\ad_\alpha=\tr_\alpha$  for every bi-transposition $\alpha$, one has 
$L_{(k\ell)} \leq B_{(ij)(k\ell)} = C_{(ij)(k\ell)} \leq H_{(ij)}$.
If $q \neq 2$, then taking $(x, y) \perp \{i, j, k, \ell\}$, one has by $2$-divisibility $L_{(k\ell)} \leq  C_{(ij)(xy)} \leq H_{(ij)}.$
In either case, $V_{\{i, j\}^\perp} \leq H_{(ij)}$. 
Moreover, any two $\alpha, \beta \in E_{i, j}$ agree modulo $[A_{\{i, j\}^\perp}, V]$, so using Claim~\ref{elemma:cl:V}\eqref{i:bracketAI},
\[C_\beta \leq C_\alpha + [A_{\{i, j\}^\perp}, V] = C_\alpha + V_{\{i, j\}^\perp}.\] 
Thus, we have in fact shown  $H_{(ij)} = C_\alpha + V_{\{i, j\}^\perp}$ for any $\alpha \in E_{i, j}$.

Fix distinct $k,a,b \notin \{i,j\}$. Since $V_{\{i, j\}^\perp} \leq H_{(ij)}$, Claim~\ref{elemma:cl:V}(\ref{i:VItotal},\ref{i:sumIJ}) imply $V =  V_{\{i,j,k\}} + H_{(ij)}$.
Also, for  $\alpha = (ij)(ab)$, by Claim~\ref{elemma:cl:LijProperties}\eqref{i:localequations} one has $\ad_\alpha (V_{\{i,j,k\}}) \le L_{(ij)}$, and \[\ad_\alpha H_{(ij)} = \ad_\alpha(C_\alpha + V_{\{i,j\}^\perp}) = \ad_\alpha(\underbrace{C_\alpha + V_{\{i,j,a,b\}^\perp}}_{\le C_V(\alpha)} + V_{\{k,a,b\}}) \le V_{\{k,a,b\}}.\] By Claim~\ref{elemma:cl:V}\eqref{i:LVI}, $L_{(ij)} \cap V_{\{k,a,b\}} = 0$, so $\ad_\alpha(V_{\{i, j, k\}} \cap H_{(ij)} = 0$, meaning $V_{\{i,j,k\}} \cap H_{(ij)} \leq C_V(\alpha)$.
\end{proofclaim}
\begin{remark}
If $q \neq 2$, then one even has $V = H_{(ij)} + L_{(ij)}$, while if $q = 2$ then $L_{(ij)} \leq H_{(ij)}$. We however give a characteristic-independent endgame, treating reflections and transvections at once.
\end{remark}

\begin{claim}\label{elemma:cl:Embed}
There are compatible involutive operators $\{\tau_{(ij)}: i \neq j\} \subset \Aut_\cU(V)$ such that for $\Sigma := \generated{\tau_{(ij)}: i \neq j}$ and $S:= \Sym(n)$ we have:
\begin{itemize}
\item 
the map $(ij) \mapsto \tau_{(ij)}$ extends to an isomorphism $S \simeq \Sigma$ with $A = \Sigma'$;
\item
the image of $S_{(ij)^\perp}$ centralises $B_{\tau_{(ij)}}$.
\end{itemize}
\end{claim}
\begin{proofclaim}
Fix any $k \notin \{i,j\}$. By Claim~\ref{elemma:cl:V}\eqref{i:AIperponVI}  $V_{\{i,j,k\}}$ is centralised by $A_{(ijk)^\perp}$, so every choice of $\alpha = (ij)(ab)$ with $a,b\notin \{i,j,k\}$ yields the same action on $V_{\{i,j,k\}}$. 
Using Claim~\ref{elemma:cl:H}, define $\tau_{(ij)}$ (currently depending on $k$) to agree with any such $\alpha$ on $V_{\{i,j,k\}}$ while centralising $H_{(ij)}$, which  makes sense as the intersection lies in $C_V(\alpha)$. Also, notice how $[\tau_{(ij)},V] = [\tau_{(ij)},V_{\{i,j,k\}}] = [\alpha,L_{(ik)} + L_{(jk)}] =  L_{(ij)}$ by Claim~\ref{elemma:cl:LijProperties}(\ref{i:localequations},\ref{i:linegeometry}).

We claim that our definition of $\tau_{(ij)}$ does not depend on the choice of $k$. Fix distinct  $k,k' \notin \{i,j\}$, with corresponding $\tau_{(ij)}$ and $\tau'_{(ij)}$. Consider the decomposition given by Claim~\ref{elemma:cl:V}(\ref{i:VItotal},\ref{i:sumIJ}):
\[V = V_{\{i,j,k,k'\}} + V_{\{i,j\}^\perp}.\]
Note that $V_{\{i,j,k,k'\}} = V_{\{i,j,k\}} + L_{(kk')} = V_{\{i,j,k'\}} + L_{(kk')}$. 
For $a,b\notin \{i,j,k,k'\}$, $\tau_{(ij)}$ and $(ij)(ab)$ agree on $V_{\{i,j,k\}}$ (by definition) and on $L_{(kk')}$, which they both centralise since $L_{(kk')}\leq V_{\{i, j\}^\perp} \leq H_{(ij)}$ by Claim~\ref{elemma:cl:H}. Hence $\tau_{(ij)}$ and $(ij)(ab)$ agree on $V_{\{i,j,k,k'\}}$. The same holds for $\tau'_{(ij)}$, so $\tau_{(ij)}=(ij)(ab)= \tau'_{(ij)}$ on $V_{\{i,j,k,k'\}}$. And as  $\tau_{(ij)}$ and $\tau'_{(ij)}$ centralise $V_{\{i,j\}^\perp} \le H_{(ij)}$, they agree globally. This completes the definition of the transpositions; they are in $\cU$ by the existence of compatible extensions from the \hyperlink{h:uproperties}{Universe Properties}.

We now verify that $\tau_{(ij)}\tau_{(jk)} = (ijk)$. Fix distinct $\ell, a,b \notin \{i,j,k\}$. Similar to as above (with $\ell$ replacing $k'$), consider \[V = V_{\{i,j,k,\ell\}} + V_{\{i,j,k\}^\perp}.\]
Since we may define $\tau_{(ij)}$ and $\tau_{(jk)}$ using the common parameter $\ell$, we find as before that $\tau_{(ij)}$ acts on $V_{\{i,j,k,\ell\}} = V_{\{i,j,\ell\}} + L_{(k\ell)} = V_{\{j,k,\ell\}} + L_{(i\ell)}$ as $(ij)(ab)$ and $\tau_{(jk)}$ as $(jk)(ab)$. 
Since each of $\tau_{(ij)}$, $\tau_{(jk)}$, and $(ijk)$  centralise  $V_{\{i,j,k\}^\perp}$, we are done.

These relations (which also imply that disjoint transpositions commute) guarantee $\Sigma \simeq \Sym(n)$ with $\Sigma' = A$. As noted, $B_{\tau_{(ij)}} = L_{(ij)}$, so distinct transpositions give rise to disjoint brackets, which clearly implies that the image of $S_{(ij)^\perp}$ centralises $B_{\tau_{(ij)}}$.
\end{proofclaim}
This proves existence and it remains to deal with uniqueness.

\begin{claim}\label{l:elemma:cl:uniqueness}
If $S \le \Aut_\cU(V)$ represents any compatible action of $\Sym(n)$ extending $\Alt(n)$, then up to tensoring with the signature, $S = \Sigma$.
\end{claim}
\begin{proofclaim}
Say that $S$ is generated by transpositions $t_{(ij)}$. We prove that, up to tensoring, each transposition $t_{(ij)}$ coincides with $\tau_{(ij)}$ as defined in Claim~\ref{elemma:cl:Embed}. Clearly for any $s \in S$ one has $s(L_{(ij)}) = L_{(s(i) s(j))}$.

Consider the action of $t = t_{(ij)}$ on $L_t = L_{(ij)}$; write $L_t^+ = \tr_t(L_t)$.
Of course $\ell = \dim L_t$ and $\ell^+ = \dim L_t^+$ do not depend on $t$; as a matter of fact for $s \in S$ we still have $s(L_{(ij)}^+) = L_{(s(i)s(j))}^+$.
We shall prove that up to tensoring with the signature, $L_t^+ = 0$ for any transposition $t \in S$. 

First, we show that, up to tensoring, we may assume $2\ell^+ \leq \ell$. Here is an argument we shall repeat in the proof of the \hyperlink{h:fglemma}{First Geometrisation Lemma}, Claim~\ref{c:fgl:cl:tensoring}. 
In characteristic $2$, one has $(t+1)^2 = 0$, so $\im (1 + t) \leq \ker(1+t)$. Also, \emph{restricting $t$ to $L_t$}, we have that $\dim \ker(1 + t) + \dim \im (1 + t) = \dim L_t = \ell$; hence $2 \ell^+ \leq \ell$. 
In characteristic not $2$, up to tensoring, we may exchange $\tr_t$ with $\ad_t$, hence $\tr_t(L_t)$ with $\ad_t(L_t)$, and therefore assume the same. So up to tensoring (to no effect if $q = 2$), $2\ell^+ \leq \ell$. Under this assumption, we  prove  $t_{(ij)} = \tau_{(ij)}$, viz.~uniqueness.

We first contend $L_{(12)}^+ \leq L_{(13)}^+ + L_{(23)}^+$.
Let $a \in L_{(12)}^+$; by Claim~\ref{elemma:cl:LijProperties}\eqref{i:linegeometry} there is a decomposition $a = a_{13} + a_{23}$ with $a_{13} \in L_{(13)}$ and $a_{23} \in L_{(23)}$. Notice how $a = t_{(12)} a = t_{(12)} a_{23} + t_{(12)}a_{13}$, so as $L_{(13)} \cap L_{(23)} = 0$, we find that  $t_{(12)}$ swaps $a_{13}$ and $a_{23}$.
Our approach is to show that we similarly have $t_{(13)} a = a_{23}$ and  $t_{(23)} a = a_{13}$; then, as $t_{(13)} (L_{(12)}^+) = L_{(23)}^+$ and $t_{(23)}(L_{(12)}^+) = L_{(13)}^+$, we will be done. Now, using $a_{23} \in L_{(23)} \leq B_{(123)} \leq \ker \tr_{(123)}$ from Claim~\ref{elemma:cl:LijProperties}\eqref{i:localequations}:
\begin{align*}
0 & = a_{23} + (123) a_{23} + (123)^2 a_{23}\\
& = a_{23} + t_{(13)}t_{(12)} a_{23} + (132)a_{23}\\
& = a_{23} + t_{(13)} a_{13} + (132)a_{23}\\
& = \underbrace{a_{23}}_{\in L_{(23)}} + \underbrace{t_{(13)} a}_{\in L_{(23)}} - \underbrace{t_{(13)} a_{23}}_{\in L_{(12)}} + \underbrace{(132) a_{23}}_{\in L_{(12)}}
.
\end{align*}
Disjointness of lines proves $t_{(13)}a = a_{23}$, and similarly we find  $t_{(23)} a = a_{13}$. Thus, $a_{23} \in L_{(23)}^+$ and $a_{13} \in L_{(13)}^+$, as desired.

We claim that $t_{(ij)}$ inverts $L_{(ij)}$ and centralises $V_{\{i,j\}^\perp}$. 
Let $V^+ := \sum_{i \neq j} L_{(ij)}^+$. 
It is $\Alt(n)$-invariant so  equals $0$ or $V$ by dc-irreducibility. And from our work above, $V^+ = L_{(12)}^+ + L_{(23)}^+ + \dots + L_{(n-1, n)}^+$.
Since we are assuming $2\ell^+ \leq \ell$, Claim~\ref{elemma:cl:V}(\ref{i:VItotal},\ref{i:n-2lines}) now yields:
\[\dim V^+ \leq (n-1) \ell^+ \leq \frac{(n-1)}{2} \ell < (n-2) \ell \leq \dim V,\]
so $V^+ < V$, implying $V^+ = 0$.
In particular,  $t_{(ij)}$ inverts $L_{(ij)}$, and then by Claim~\ref{elemma:cl:L}, we see that for $(k\ell)\perp(ij)$,  $t_{(k\ell)} = t_{(ij)}(ij)(k\ell)$ centralises $L_{(ij)}$.

We now prove that $t_{(ij)}$ and $\tau_{(ij)}$ agree everywhere. By Claim~\ref{elemma:cl:V}(\ref{i:VItotal},\ref{i:sumIJ}), $V = V_{\{i,j,k\}} + V_{\{i,j\}^\perp}$, with both maps centralising the latter term. By definition, $\tau_{(ij)}$ acts on $V_{\{i,j,k\}}$ as $(ij)(ab)$ for any $(ab)\perp(ij)$, and as $t_{(ab)}$ centralises $V_{\{i,j,k\}}$, $t_{(ij)}$ also acts on $V_{\{i,j,k\}}$ as $(ij)(ab)$. We are done.
\end{proofclaim}

This completes the proof of the \hyperlink{h:elemma}{Extension Lemma}.
\end{proof}\setcounter{claim}{0}

\begin{remark}
Uniqueness does not hold without assuming compatibility of the action of $S$: one could break $L = L^+ \oplus L^-$ into abstract pieces (e.g.~viewing the complex field as a real vector space), producing a non-compatible decomposition $V = V^+ \oplus V^-$, one being the sign-tensored version of the other.
\end{remark}

\subsection{The Geometrisation Lemma}\label{s:glemma}
Here we establish that $\Sym(n)$- and $\Alt(n)$-modules of sufficiently low dimension necessarily carry the geometric structure of the standard module, leading to their identification by our previous work. 

We begin by treating the minimal case for $\Sym(n)$. The authors disagree on the significance and relevance in the present paper of isolating this situation; it is to the reader to decide. Nevertheless, the \hyperlink{h:fglemma}{First Geometrisation Lemma} will be used in the proof of the significantly more powerful \hyperlink{h:glemma}{Geometrisation Lemma}. 
Do note that the special case addressed here holds for all $n$ (trivially  for $n<3$). As a result, we need to account for $\Sym(6)$-modules that are `quasi-equivalent' to $\rstd(6, 2)$ via a twisting of the module by an outer automorphism of $\Sym(6)$.

\begin{fglemma}
Let $n\ge 3$,  $S:=\Sym(n)$, and $V \in \Mod(S, d, q)$ be faithful. Then 
$d \geq n-2$; if equality holds, then $n\ge 5$, and for any transposition $\tau \in S$, we are in one of the following cases:
\begin{itemize}
\item
up to tensoring with the signature, $V$ satisfies  $[S'_{\tau^\perp}, B_\tau] = 0$;
\item
$n = 6$, $q=2$, and up to composing the action with some $\sigma \in \Out(\Sym(6))$, $V$ satisfies  $[S'_{\tau^\perp}, B_\tau] = 0$. 
\end{itemize}
\end{fglemma}

\begin{remarks}\leavevmode
\begin{itemize}
\item
The (reduced) standard module is \emph{not} faithful when $n = 3, 4$.
\item
The proof requires going to quotients, so dealing with objects which are not subobjects. See Remarks on p.~\pageref{r:catenarity}.
\end{itemize}
\end{remarks}

\begin{proof}
Let $S = \Sym(n)$ and $A = S'$. 
We begin with the soluble cases. 

\begin{claim}\label{c:fgl:cl:soluble}
If $n\in \{3,4\}$, then $d\ge n-1$. 
\end{claim}
\begin{proofclaim}
Suppose $d \leq n-2$. If $n = 3$, then the action cannot be faithful since involutions must either centralise or invert a $1$-dimensional dc-module, forcing $A$ to act trivially. 

If $n = 4$, then for the same reason $d = 2$, and if $q \neq 2$, then the \hyperlink{h:wlemma}{Weight Lemma} yields a contradiction. So if $n = 4$, then $q = 2$, and involutions act quadratically (with $(i-1)^2 = 0$). By faithfulness, for bi-transpositions $\alpha \neq \beta$, one has $\dim B_\alpha = \dim B_\beta = 1$. Thus, $B_\alpha = (C_V(\alpha))^1$, and letting $\alpha$ act on  $B_\beta$, one finds $B_\beta \le (C_V(\alpha))^1$, forcing $B_\alpha = B_\beta$. Hence, $B_\alpha = [A',V]$ is $S$-invariant, but then $A$ centralises the $1$-dimensional dc-modules $[A',V]$ and $V/[A',V]$, contradicting faithfulness (and non-nilpotence of $A$).
\end{proofclaim}

From now on $n \geq 5$. We may thus also suppose that $d$ is minimal such that $V$ is faithful; consequently, 
$V$ is now dc-irreducible. 

Assume $d \leq n-2$; we shall identify the module. 
Let $b = \dim B_\tau$, which does not depend on the transposition $\tau$.

\begin{claim}[`up to tensoring']\label{c:fgl:cl:tensoring}
We may assume $2b \leq d$.
\end{claim}
\begin{proofclaim}
If $q = 2$ then $\tau$ acts quadratically, viz.~$(\tau-1)^2 = 0$. Hence $B_\tau = \im (1 - \tau) \leq \ker(1-\tau)$. On the other hand, $\dim \ker(1 - \tau) + \dim \im (1 - \tau) = d$; hence $2b \leq d$. (This argument already appeared in the \hyperlink{h:elemma}{Extension Lemma}, Claim~\ref{l:elemma:cl:uniqueness}; it will be used again in the \hyperlink{h:glemma}{Geometrisation Lemma}, Claim~\ref{l:glemma:cl:even}.) If $q \neq 2$ then by the \hyperlink{h:clemma}{Coprimality Lemma}, $V = B_\tau \qoplus C_\tau$.  Tensoring with $\varepsilon$ exchanges $\tau$ with $-\tau$, hence $B_\tau$ with $C_\tau$, so up to tensoring, we may assume $\dim B_\tau \leq \dim C_\tau$ and $2b \leq d$.
\end{proofclaim}

\begin{claim}
If $n\in \{5,6\}$, we are done. 
\end{claim}
\begin{proofclaim}
Let $n\le 6$ and $\tau$ be a transposition. If $b=1$, then $S_{\tau^\perp}$ must centralise or invert $B_\tau$, forcing $A_{\tau^\perp}$ to centralises $B_\tau$, so we may assume $b\ge 2$. Since $d \ge 2b$, we have reduced to the case of $n = 6$, $d = 4$, and $b = 2$. Moreover, by Claim~\ref{c:fgl:cl:soluble}, $S_{\tau^\perp} \simeq \Sym(4)$ is faithful on neither $B_\tau$ nor $V/B_\tau$, so bi-transpositions in $S_{\tau^\perp}$ are quadratic. Thus $q = 2$, and it remains to deal with this exotic configuration.

We  claim that for any $3$-cycle $\gamma$ one has $B_\gamma = V$.
We have seen that $K_{\tau^\perp} = A_{\tau^\perp}'$ is trivial on each of $B_\tau$ and $V/B_\tau$, so on each factor, all $3$-cycles from $S_{\tau^\perp}$ have the same action. Let $\gamma$ be one such $3$-cycle. Notice how $B_\tau = \ad_\gamma(B_\tau) \qoplus \tr_\gamma(B_\tau)$ is an $S_{\tau^\perp}$-invariant decomposition, so if either factor  is $1$-dimensional, then $\gamma \in A_{\tau^\perp}'$ centralises both, a contradiction. So, $\ad_\gamma(B_\tau)$ has dimension $0$ or $2$, and the analogous statement holds (with analogous proof) for $V/B_\tau$. Thus, $\dim B_\gamma$ is $2$ or $4$.  If $\gamma$ centralises $V/B_\tau$ then $B_\gamma \leq B_\tau$  and equality holds. Then $B_\gamma = B_\tau = B_{\gamma'}$ for any other $3$-cycle in $S_{\tau^\perp}$; conjugating, we contradict dc-irreducibility. If $\gamma$ centralises $B_\tau$ we find $C_\gamma = B_\tau = C_{\gamma'}$, again a contradiction. This shows $B_\gamma = V$.

Now let $\gamma, \delta$ be disjoint $3$-cycles and $\varepsilon = \gamma \delta$. If $B_\varepsilon = V$ as well, then $\tr_\gamma = \tr_\delta = \tr_\varepsilon = 0$, forcing:
\begin{align*}
0 & = 1 + \varepsilon + \varepsilon^2\\
& = 1 + \gamma \delta + (-1 - \gamma) (-1 - \delta)\\
& = \gamma + \delta,
\end{align*}
a clear contradiction. So $B_\varepsilon < V$. Since the outer (class of) automorphism swaps (the class of) $\gamma$ with (that of) $\varepsilon$, we deduce that twisting by $\sigma$, we no longer have $b = 2$. Hence, in this exotic case as well, we reduced to $b = 1$, implying that $A_{\tau^\perp}$ centralises $B_\tau$. 
\end{proofclaim}

\begin{claim}
If $n\ge 7$, we are done. 
\end{claim}
\begin{proofclaim}
Let $\tau$ be a transposition. By induction, $A_{\tau^\perp}$ centralises $B_\tau$, as the alternative implies $n-4 \le b \le \frac{1}{2}(n-2)$. 
\end{proofclaim}
\setcounter{claim}{0}
This completes the \emph{first} geometrisation argument.
\end{proof}

We now present our main geometrisation result. 

\begin{glemma}
Let $A = \Alt(n)$ and $V \in \Mod(A, d, q)$. Suppose  $d < n$ and  
that either
\begin{itemize}
\item
$q = 2$ and $n \geq 10$; or
\item
$q \neq 2$ and $n \geq 7$.
\end{itemize}
Then for any bi-transposition $\alpha \in A$ one has $[A_{\alpha^\perp}, B_\alpha] = 0$.
\end{glemma}


\begin{proof}
We want to show that $A_{\alpha^\perp}$ acts trivially on $B_\alpha$; by conjugacy the desired property does not depend on $\alpha$.
The proof will by induction on $n$, but methods depend on the value of $q$.
Let $b = \dim B_\alpha$.

\begin{claim}[odd case]\label{l:glemma:cl:odd}
If $q \neq 2$, then we are done.
\end{claim}
\begin{proofclaim}
Let $K = \{1, \alpha_i, \alpha_j, \alpha_k\}$ be a four-group of bitranspositions on the same support. Let $\Lambda = \{\lambda_0, \lambda_i, \lambda_j, \lambda_k\}$ be the weights with $\lambda_0 = 1$ and $\lambda_i = \alpha_i^\vee$, viz.~$\lambda_i(\alpha_i) = 1$ while $\lambda_i(\alpha_j) = \lambda_i(\alpha_k) = -1$. By the \hyperlink{h:wlemma}{Weight Lemma} there is a decomposition, in obvious notation,
\[V = \bigqoplus_{\lambda\in \Lambda} V_\lambda.\]
Notice at once $B_{\alpha_i} = V_{\lambda_j} + V_{\lambda_k}$.
Of course, $A_{\alpha^\perp}\simeq \Alt(n-4)$ normalises all weight spaces. Actually, we have more. There is $A_{\alpha^\perp} < \Sigma_i < C_A(\alpha_i)$ with $\Sigma_i\simeq \Sym(n-4)$; the group $\Sigma_i$ normalises $V_{\lambda_0}$ and $V_{\lambda_i}$ while swapping $V_{\lambda_j}$ and $V_{\lambda_k}$.
(Typically, if $\alpha_i = (12)(34)$, we take $\Sigma_i = \Sigma_{\alpha^\perp}^{(12)}$, which exchanges $\alpha_j = (13)(24)$ and $\alpha_k = (14)(23)$, hence also the relevant weights.)

Now, suppose that $A_{\alpha^\perp}$ does \emph{not} centralise $B_\alpha$. Then there is $i$ such that $A_{\alpha^\perp}$ does not centralise $V_{\lambda_i}$. The action extends to one of $\Sigma_i$. 

Let $a = \dim V_{\lambda_i}$, which does not depend on $i \neq 0$. As we are assuming $[A_{\alpha^\perp},V_{\lambda_i}] \neq 0$, we find $1 < a 
 \leq \frac{d}{3}$.
\begin{itemize}
\item
The case $n = 7$ requires scrutiny. The only nontrivial subcase is when $d = 6$, $a = 2$, and $V = V_{\lambda_i} + V_{\lambda_j} + V_{\lambda_k}$. Also $\Sigma_i\simeq \Sym(3)$ acts on $V_{\lambda_i}$, and the action of $\Sigma_i' = A_{\alpha^\perp}$ is nontrivial.
Therefore an involution $\tau \in \Sigma_i$ will satisfy $\dim [\tau, V_{\lambda_i}] = 1$.
As $\Sigma_i$ swaps $V_{\lambda_j}$ and $V_{\lambda_k}$, which have dimension $2$, $\dim [\tau, V_{\lambda_j} + V_{\lambda_k}] = 2$.
As a conclusion, $\dim B_\tau(V) = 3$, but $\tau$ is a bitransposition of $A$, hence $\dim B_\tau = 3 = 2a$, a contradiction.
\item
The same argument removes $n = 8$. (A possible $1$-dimensional $V_{\lambda_0}$ causes no problem.)
\item
If $n\ge 9$, the action of $\Sigma_i$ on $V_{\lambda_i}$ is faithful, and the \hyperlink{h:fglemma}{First Geometrisation Lemma} yields $n-6 \le a \le \frac{d}{3} \le \frac{1}{3}(n-1)$, a contradiction.
\end{itemize}
Hence $A_{\alpha^\perp}$ centralises $B_\alpha$, as claimed.
\end{proofclaim}

\begin{remark}
Assuming $V$ is faithful, one can even show $a = 1$ fairly directly; however it also is a consequence of applying \hyperlink{h:elemma}{Extension} and \hyperlink{h:rlemma}{Recognition}.
\end{remark}

\begin{claim}[even case]\label{l:glemma:cl:even}
If $q = 2$ then we are done.
\end{claim}
\begin{proofclaim}
Suppose the action of $A_{\alpha^\perp}$ on $B_\alpha$ is \emph{not} trivial. As in Claim~\ref{l:glemma:cl:odd}, this action extends to a faithful action of some $\Sigma \simeq \Sym(n-4)$ with $A_{\alpha^\perp} < \Sigma < C_A(\alpha^\perp)$.

Let $b:= \dim B_\alpha$, for $\alpha$ a bitransposition; by quadraticity, $b \leq \frac{d}{2}$. (See the proof of the \hyperlink{h:fglemma}{First Geometrisation Lemma}, Claim~\ref{c:fgl:cl:tensoring}).
\begin{itemize}
\item
The case $n= 10$ requires a close look. 
Here, $b \le 4$.

Observe that $\Sigma$ is also faithful on $V/B_\alpha$. If not, then any disjoint bi-transposition $\beta \perp \alpha$ will satisfy $B_\beta \leq B_\alpha$, whence equality; conjugating, we find that $B_\alpha$ is $A$-invariant and centralised by all bi-transpositions, a contradiction. Hence  $B_\alpha,V/B_\alpha\in\Mod(\Sigma)$ are both faithful. 

By the \hyperlink{h:fglemma}{First Geometrisation Lemma}, the $\Sigma$-module $B_\alpha$ is known (be careful that $\dim (V/B_\alpha)$ could be $5$). Now choose $\gamma = (ijk) \in A_{\alpha^\perp}$. Inspection of the possible structures for $B_\alpha \in \Mod(\Sigma, 4, 2)$ gives that $[\gamma,B_\alpha]$ has dimension $2$ if $B_\alpha$ is the standard module; for the exotic twist, the dimension is $3$.
Looking at $V/B_\alpha$, we certainly have $\dim [\gamma,V/B_\alpha] \ge 2$;  otherwise, for any $(ij)(ab)\in A_{\alpha^\perp}$ inverting $\gamma$, the group $\langle \gamma, (ij)(ab)\rangle \simeq \Sym(3)$ does not act faithfully on the $1$-dimensional $[\gamma,V/B_\alpha]$, a contradiction since $q\neq3$.
Using $\dim B_\gamma = \dim [\gamma, V] = \dim [\gamma, B_\alpha] + \dim [\gamma, V/B_\alpha]$ by the \hyperlink{h:clemma}{Coprimality Lemma}, we find that $4 \leq \dim B_\gamma < d$. 

Turn to the action of $\Sigma_{\gamma^\perp}^{(ij)} \simeq \Sym(7)$ on both $B_\gamma$ and $C_\gamma$. It must be faithful on at least one, so by the \hyperlink{h:fglemma}{First Geometrisation Lemma}, one has dimension at least $6$ (and the other at most $3$). All this shows $\dim B_\gamma \ge 6$ and $\dim C_\gamma \le 3$. In particular $A_{\gamma^\perp}$ centralises $C_\gamma$, so $C_\gamma = C_{\gamma'}$ for any $3$-cycle $\gamma'\perp\gamma$. Conjugating, we find a contradiction.
\item
Assume $n\ge 11$. The action of $\Sigma$ on $B_\alpha$ is faithful, so the \hyperlink{h:fglemma}{First Geometrisation Lemma} yields $n-6 \le b \le \frac{d}{2} \le \frac{1}{2}(n-1)$. This is an immediate contradiction when $n\ge 12$, and when $n=11$, the \hyperlink{h:fglemma}{First Geometrisation Lemma} further implies that $q$ divides $n-4 = 7$, again a contradiction.\qedhere
\end{itemize}
\end{proofclaim}

This completes the proof of the \hyperlink{h:glemma}{Geometrisation Lemma}.
\setcounter{claim}{0}
\end{proof}

\begin{remark}
Although we are not able to identify the minimal faithful $V \in \Mod(\Alt(9), d, 2)$ (bearing in mind that there are three such in the classical setting), we can still easily show that the minimal dimension is $n-1$. Indeed, with notation as above, notice first that if the action of $A_{\alpha^\perp}$ on $B_\alpha$ is trivial, then \hyperlink{h:elemma}{Extension} followed by \hyperlink{h:rlemma}{Recognition} shows that $d\ge n-1$ (since $2\nmid 9$). And if the action of $A_{\alpha^\perp}$ on $B_\alpha$ is \emph{not} trivial, the \hyperlink{h:fglemma}{First Geometrisation Lemma} (applied to $\Sigma \cong \Sym(5)$) shows that $b\ge 4$ (since $2\nmid 5$), so by quadraicity, $4\le b \leq \frac{d}{2}$, as desired.
\end{remark}

\subsection{Assembling the Theorem}\label{s:theorem}


\begin{proof}[Proof of the \hyperlink{h:theorem}{Theorem}]
Let $S:=\Sym(n)$ and $A := \Alt(n)$.
Suppose $G = A $ or $S$ and $V \in \Mod(G, d, q)$ is faithful and dc-irreducible with $d < n$. We always assume $n\ge 7$; if $G = \Alt(n)$ and $q=2$, we further assume $n \geq 10$.

We first assume that $V$ is dc-irreducible as an $A$-module. In this case, we are done except when $V \in \Mod(S, d, 2)$ with $7\le n\le9$. This follows immediately from applying (in order) the \hyperlink{h:glemma}{Geometrisation Lemma}, the \hyperlink{h:elemma}{Extension Lemma}, and the \hyperlink{h:rlemma}{Recognition Lemma}. We are also using that $d < n$ forces $\dim L = 1$ (with notation as in the \hyperlink{h:rlemma}{Recognition Lemma}) and $d = n-2$ or $n-1$, according to the value of $q$.

We next address when $V \in \Mod(S, d, 2)$, aiming to show $A_\tau^\perp$ centralises $B_\tau$ for $\tau\in S$ a transposition. 
Set $b:= \dim B_\tau$; as usual, by quadraticity, $2b\le d$.  If $A_\tau^\perp \cong \Alt(n-2)$ does not centralises  $B_\tau$, then the \hyperlink{h:fglemma}{First Geometrisation Lemma} (applied to $S_\tau^\perp$) implies that  $b\ge n-4$, with equality only possible if $2\mid n-2$. Thus, $2(n-4) \le 2b \le d \le n-1$, so $n=7$ and $b=n-4$. But then $2\mid 5$. Thus $A_\tau^\perp$ centralises $B_\tau$, and the \hyperlink{h:rlemma}{Recognition Lemma} applies.

So it remains to dispose of the case when $V$ is dc-irreducible as an $S$-module but not as an $A$-module. Assume this is the case, and consider a dc-irreducible $A$-series for $V$. By our work above, each factor is either a trivial module or is standard, and by simplicity of $A$, some factor is nontrivial, hence of dimension at least $n-2$.  Thus $d = n-1$, and there must be an $A$-submodule $W \le V$ such that either $\dim W = n-2$ and $\dim V/W = 1$ or vice versa.
The former  implies $W = [A,V]$, against dc-irreduciblity of $V$ as an $S$-module. The latter implies $W \le C_V(A)$, hence equality by the structure of $V/W$, again violating dc-irreduciblity as an $S$-module.
\end{proof}

\section*{Acknowledgements}\label{S:acknowledgements}

This is: Paris Album No.~2.
Research began in  early Summer  2018 in Paris by the last two authors who treated the lower bound (but not recognition) for $\Sym(n)$; in Fall 2018, the first two met in Bogot\'{a} and the last two in Bonn to attack $\Alt(n)$. All three authors met in Paris in early Summer 2019 to clarify the approach;  \hyperlink{h:rlemma}{Recognition} and \hyperlink{h:elemma}{Extension}  (as they now appear) evolved throughout 2020. The general context  was only finalized recently.

The second author wishes to thank both the \textsc{cimpa} network (Centre international de math\'{e}matiques pures et appliqu\'{e}es) and the \textsc{ecos}-Nord programme for sponsoring visits to Universitad de los Andes in Bogot\'{a}. He also thanks the Hausdorff Institute for Mathematics (HIM) in Bonn, and Fromagerie Quatrehomme for providing cheese during all Paris stays and trips abroad.

The third author also benefited immensely from the generous support of HIM and Fromagerie Quatrehomme. Additionally, he is  grateful to the Oberwolfach Research Institute for Mathematics (MFO) for  his stay in January 2020 during the program on \textit{Model Theory: Groups, Geometries and Combinatorics}; it was there that the \hyperlink{h:rlemma}{Recognition Lemma} was born.
The work of the third author was partially supported by the National Science Foundation under grant No. DMS-1954127 as well as by 
 the Creative Activity Faculty Awards Program and the College of Natural Science and Mathematics at California State University, Sacramento.

All authors wish to thank Tuna Alt\i{}nel, the first colleague with whom we discussed the new `dimensional' setting.
A Turkish mathematician and ma\^i{}tre de conf\'erences at Universit\'e Claude Bernard Lyon~1, Alt\i{}nel remained a \emph{de facto} prisoner of the Republic of Turkey for over two years. His passport was locked up by the Turkish authorities from April 2019 to June 2021; for 81 of those days, he was as well. 
The second author met with Alt{\i}nel in \.{I}stanbul in January 2020; Gregory Cherlin was also there. The occasion was a hearing in Alt\i{}nel's trial: but high spirits always welcome research seminars, despite political prosecution and judicial harassment. Thanks to international protest, including uprise from various scientific societies, Alt{\i}nel was acquitted in two distinct cases, all appeals were settled by September 2020, and, with incredible delay, his passport was then returned over nine months later.
 However, his freedom of travel is still under under judicial threat. Details on \href{http://math.univ-lyon1.fr/SoutienTunaAltinel/}{math.univ-lyon1.fr/SoutienTunaAltinel}. 

Paris Album No.~2 is dedicated to those laboring for peace 
and 
freedom of speech.


\bibliographystyle{alpha}
\bibliography{SymBib}

\newcommand{\etalchar}[1]{$^{#1}$}
\newcommand{\noopsort}[1]{}\def\rasp{\leavevmode\raise.45ex\hbox{$\rhook$}}
  \def\cprime{$'$}
\begin{thebibliography}{ABL{\etalchar{+}}05}

\bibitem[ABL{\etalchar{+}}05]{atlas3}
Rachel Abbott, John Bray, Steve Linton, Simon Nickerson, Simon Norton, Richard
  Parker, Ibrahim Suleiman, Jonathan Tripp, Peter Walsh, and Robert Wilson.
\newblock Atlas of finite group representations, version 3.
\newblock Available online at {\url{http://brauer.maths.qmul.ac.uk/Atlas/v3}},
  2005.

\bibitem[AW18]{AlWi18}
Tuna Alt{\i}nel and Joshua Wiscons.
\newblock Recognizing {$\operatorname{PGL}_3$} via generic {$4$}-transitivity.
\newblock {\em J. Eur. Math. Soc. (JEMS)}, 20(6):1525--1559, 2018.

\bibitem[BB18]{BoBe18}
Ay\c{s}e Berkman and Alexandre Borovik.
\newblock Groups of finite {M}orley rank with a generically sharply multiply
  transitive action.
\newblock {\em J. Algebra}, 513:113--132, 2018.

\bibitem[BB21]{BeBo21}
Ay\c{s}e Berkman and Alexandre Borovik.
\newblock Groups of finite {M}orley rank with a generically multiply transitive
  action on an abelian group.
\newblock 15 pages. \url{https://arxiv.org/abs/2107.09997}, 2021.

\bibitem[BC08]{BoCh08}
Alexandre Borovik and Gregory Cherlin.
\newblock Permutation groups of finite {M}orley rank.
\newblock In {\em Model theory with applications to algebra and analysis.
  {V}ol. 2}, volume 350 of {\em London Math. Soc. Lecture Note Ser.}, pages
  59--124. Cambridge Univ. Press, Cambridge, 2008.

\bibitem[BD16]{BoDe15}
Alexandre Borovik and Adrien Deloro.
\newblock Rank 3 bingo.
\newblock {\em J. Symb. Log.}, 81(4):1451--1480, 2016.

\bibitem[Bor]{BoA20}
Alexandre Borovik.
\newblock Finite group actions on abelian groups of finite {M}orley rank.
\newblock Preprint, arXiv:2008.00604 [math.GR].

\bibitem[BY18]{BY18}
Alexandre Borovik and \c{S}\"{u}kr\"{u} Yal\c{c}{\i}nkaya.
\newblock Adjoint representations of black box groups {${\rm PSL}_2(\mathbb
  F_q)$}.
\newblock {\em J. Algebra}, 506:540--591, 2018.

\bibitem[Del09]{DeA09}
Adrien Deloro.
\newblock Actions of groups of finite {M}orley rank on small abelian groups.
\newblock {\em Bull. Symbolic Logic}, 15(1):70--90, 2009.

\bibitem[Del12]{Dprank}
Adrien Deloro.
\newblock $p$-rank and $p$-groups in algebraic groups.
\newblock {\em Turkish J. Math.}, 36(4):578––582, 2012.

\bibitem[Dic08]{DiL08}
Leonard~Eugene Dickson.
\newblock Representations of the general symmetric group as linear groups in
  finite and infinite fields.
\newblock {\em Trans. Amer. Math. Soc.}, 9(2):121--148, 1908.

\bibitem[Jam78]{JaG78}
G.~D. James.
\newblock {\em The representation theory of the symmetric groups}, volume 682
  of {\em Lecture Notes in Mathematics}.
\newblock Springer, Berlin, 1978.

\bibitem[Jam83]{JaG83}
G.~D. James.
\newblock On the minimal dimensions of irreducible representations of symmetric
  groups.
\newblock {\em Math. Proc. Cambridge Philos. Soc.}, 94(3):417--424, 1983.

\bibitem[Wag76]{WaA76}
Ascher Wagner.
\newblock The faithful linear representation of least degree of {$S_{n}$} and
  {$A_{n}$} over a field of characteristic {$2.$}.
\newblock {\em Math. Z.}, 151(2):127--137, 1976.

\bibitem[Wag77]{WaA77}
Ascher Wagner.
\newblock The faithful linear representations of least degree of {$S_{n}$} and
  {$A_{n}$} over a field of odd characteristic.
\newblock {\em Math. Z.}, 154(2):103--114, 1977.

\bibitem[Wag20]{WaF20}
Frank~O. Wagner.
\newblock Dimensional groups and fields.
\newblock {\em J. Symb. Log.}, 85(3):918--936, 2020.

\end{thebibliography}
\end{document}